\documentclass[a4paper]{amsart}
\usepackage{verbatim}
\usepackage[bookmarks, colorlinks=true, pdfstartview=FitV, linkcolor=blue, citecolor=blue, urlcolor=blue]{hyperref}

\newcommand{\cal}{\mathcal}

\newcommand{\rd}{{\mathbb R}^d}

\newcommand{\R}{{\mathbb R}}
\newcommand{\N}{{\mathbb N}}

\newcommand{\exo}{{\rm exo}}
\newcommand{\Unp}{{\rm Unp}}
\newcommand{\piAB}{\pi_A^{-1}(B)}
\newcommand{\piAC}{\pi_A^{-1}(C)}

\newcommand{\spt}{{\rm spt}\,}

\newcommand{\Ha}{{\cal H}}
\newcommand{\bd}{\partial}
\newcommand{\cl}[1]{\overline{#1}}

\newcommand{\ind}[1]{\mathbf{1}_{#1}}
\newcommand{\eps}{\varepsilon}
\newcommand{\sB}{\cal{B}}
\newcommand{\sC}{\cal{C}}
\newcommand{\sA}{\cal{A}}
\newcommand{\sS}{\cal{S}}
\newcommand{\sM}{\cal{M}}
\newcommand{\usS}{\overline{\cal S}}
\newcommand{\usM}{\overline{\cal M}}
\newcommand{\lsS}{\underline{\cal S}}
\newcommand{\lsM}{\underline{\cal M}}
\newcommand{\ausS}{{\overline{\widetilde{\cal S}}}}
\newcommand{\ausM}{\overline{\widetilde{\cal M}}}
\newcommand{\alsS}{\underline{\widetilde{\cal S}}}
\newcommand{\alsM}{\underline{\widetilde{\cal M}}}
\newcommand{\asS}{\widetilde{\cal{S}}}
\newcommand{\asM}{\widetilde{\cal{M}}}
\newcommand{\ldim}{\underline{\dim}}
\newcommand{\udim}{\overline{\dim}}

\newcommand{\se}{\searrow}
\newcommand{\vlim}{\stackrel{v}{\longrightarrow}}
\newcommand{\wlim}{\stackrel{w}{\longrightarrow}}

\newtheorem{Theorem}{Theorem}
\newtheorem{Proposition}[Theorem]{Proposition}
\newtheorem{Lemma}[Theorem]{Lemma}
\newtheorem{Corollary}[Theorem]{Corollary}
\newtheorem{example}[Theorem]{Example}

\newtheorem{Remark}[Theorem]{Remark}

\numberwithin{equation}{section}
\numberwithin{Theorem}{section}

\begin{document}
\title[Local Minkowski measurability]{Localization results for Minkowski contents}
\author{Steffen Winter}
\address{Karlsruhe Institute of Technology, Department of Mathematics, 76128 Karlsruhe, Germany}
\date{\today}
\subjclass[2000]{28A75, 28A80, 51M25}
\keywords{parallel set, surface area, Minkowski content, Minkowski dimension, Kneser function, metrically associated, S-content, weak convergence, vague convergence, fractals, curvature measures}
\begin{abstract}
It was shown recently that the Minkowski content
of a bounded set $A$ in $\R^d$ with volume zero can be characterized in terms of
the asymptotic behaviour of the boundary surface area of its parallel sets $A_r$ as the parallel radius $r$ tends to $0$.
Here we discuss localizations of such results. The asymptotic behaviour of the local parallel volume of $A$ relative to a suitable second set $\Omega$ can be understood in terms of the suitably defined local surface area relative to $\Omega$. Also a measure version of this relation is shown: Viewing the Minkowski content
as a locally determined measure, this measure can be obtained as a weak limit of suitably rescaled surface measures of close parallel sets. Such measure relations had been
observed before for self-similar sets and some self-conformal sets in $\R^d$. They are now established for arbitrary closed sets, including even the case of unbounded sets.

The results are based on a localization of Stach\'o's famous formula relating the boundary surface area of $A_r$ to the derivative of the volume function at $r$.
\end{abstract}
\maketitle

\section{Introduction}

Let $A$ be a bounded subset of $\rd$ and $r>0$. Denote by $d(x,A)$ the (Euclidean) distance between $A$ and a point $x\in\R^d$, and by
$$
A_r:=\{ z\in\rd:\, d(z,A)< r\}
$$
 the \emph{open $r$-parallel set} (or open $r$-neighbourhood) of $A$.
 Let $V_A(r):=\lambda_d(A_r)$ be the volume of $A_r$.
Kneser \cite{Kneser} observed that the volume function $r\mapsto V_A(r)$, $r>0$ satisfies a growth estimate which is nowadays called the Kneser property, see \eqref{eq:kneser}. Stach\'o \cite{Stacho} used this to show that for all $r>0$
the left and right derivatives $(V_A)'_-(r)$ and $(V_A)'_+(r)$ of $V_A(r)$ exist   
and established a remarkable relation to the surface area of the boundary $\bd A_r$ of $A_r$: for any $r>0$,
\begin{align} \label{eq:deriv-vol}
    \sM^{d-1}(\partial A_{r})=\frac 12\left((V_A)_{-}'(r)+(V_A)_{+}'(r)\right).
  \end{align}
Here $\sM^{d-1}$ denotes the $(d-1)$-dimensional Minkowski content, cf.~\eqref{eq:def-Mink} for the definition.
Whenever the derivative $V_A'$ exists (which is the case for all $r>0$ except countably many), the Minkowski content can be replaced by the Hausdorff measure ${\cal H}^{d-1}$ and one has
\begin{equation} \label{vol_deriv}
{\cal H}^{d-1}(\partial A_r)=V_A'(r),
\end{equation}
see \cite{RW09}. It is obvious that parallel volume as well as boundary surface area are local concepts, i.e.\ when restricted to some open set $G$ intersecting the given parallel set $A_r$, then the volume $\lambda_d(A_r\cap G)$ as well as the surface area ${\cal H}^{d-1}(\partial A_r\cap G)$ will only depend on the local structure of $A$ within a suitable neighborhood of $G$. Certainly, points of $A$ with distance to $G$ larger than $r$, will have no influence on these quantities. It is therefore a natural question, whether also the relations in \eqref{eq:deriv-vol} and \eqref{vol_deriv} can be localized in some way. It is our first aim in this note to discuss such a localization. It is based on the concept of metrically associated sets already introduced by Stach\'o \cite{Stacho}.
Essentially we will show in Section 2, that the relations  \eqref{eq:deriv-vol} and \eqref{vol_deriv}  localize when we restrict parallel volume and boundary surface area of a set $A$ to sets that are metrically associated with $A$. These new local relations do even make sense for unbounded sets $A$. Our considerations have partially been motivated by some questions arising in the study of local Minkowski contents to be discussed in a moment.

In \cite{RW09} and some subsequent papers \cite{HR12,RW13,W11-1} the limiting behaviour as $r\searrow 0$ of volume $V_A(r)$ and boundary surface area ${\cal H}^{d-1}(\partial A_r)$ was studied for arbitrary bounded sets $A\subset \R^d$
 and based on the above mentioned equations \eqref{eq:deriv-vol} and \eqref{vol_deriv} some close relations have been established between the resulting notions of Minkowski content and S-content. We briefly recall their definitions.

For $s\geq 0$, the \emph{$s$-dimensional lower and upper Minkowski content} of a compact set $A\subset\R^d$ are the numbers \begin{align}\label{eq:def-Mink}
\underline{\cal M}^s(A):=\liminf_{r\to 0} \frac{V_A(r)}{\kappa_{d-s}r^{d-s}} \quad \text{ and } \quad
\overline{\cal M}^s(A):=\limsup_{r\to 0} \frac{V_A(r)}{\kappa_{d-s}r^{d-s}},
\end{align}
respectively, where $\kappa_t:=\pi^{t/2}/\Gamma(1+\frac t2)$. (If $t$ is an integer, $\kappa_t$ is the volume of the $t$-dimensional unit ball.)  Similarly, for $0\leq s<d$, the \emph{$s$-dimensional lower and upper S-content} were introduced in \cite{RW09}, by
\[
\underline{\cal S}^s(A):=\liminf_{r\to 0} \frac{{\cal H}^{d-1}(\bd A_r)}{(d-s)\kappa_{d-s}r^{d-1-s}}
\quad \text{ and } \quad
\overline{\cal S}^s(A):=\limsup_{r\to 0} \frac{{\cal H}^{d-1}(\bd A_r)}{(d-s)\kappa_{d-s}r^{d-1-s}},
\]
respectively, and for $s=d$ by $\lsS^d(A)=\usS^d(A):=0$ (justified by the fact that $\lim_{r\to 0}r{\cal H}^{d-1}(\bd A_r)=0$ for any bounded set $A\subset\R^d$, cf.~\cite{Kneser}).
If ${\lsM}^s(A)={\usM}^s(A)$, then the common value ${\cal M}^s(A)$ is referred to as the \emph{$s$-dimensional Minkowski content} of $A$ and if ${\lsS}^s(A)={\usS}^s(A)$, then similarly the common value $\sS^s(A)$ is called \emph{$s$-dimensional S-content} of the set $A$. Further, if ${\cal M}^s(A)$ ($\sS^s(A)$) exists and is positive and finite, then the set $A$ is called \emph{$s$-dimensional Minkowski measurable} (\emph{$s$-dimensional S-measurable}, respectively).
The numbers
\[
\underline{\dim}_M A:=\inf\{t\ge 0 : \lsM^t(A)=0\}=\sup\{t\ge 0 :\underline{\cal M}^t(A)=\infty\}
\]
and
\[
\overline{\dim}_M A=\inf\{t\ge 0 :\overline{\cal M}^t(A)=0\}=\sup\{t\ge 0 :\overline{\cal M}^t(A)=\infty\}
\]
are usually called the \emph{lower} and \emph{upper Minkowski dimension} of $A$ and \emph{lower} and \emph{upper S-dimension} $\ldim_S A$ and $\udim_S A$ of $A$ are defined analogously with $\lsM^t(A)$ and $\usM^t(A)$ replaced by $\lsS^t(A)$ and $\usS^t(A)$,respectively.
Minkowski measurability plays an important role for instance in connection with the modified Weyl-Berry conjecture, see e.g.\ \cite{FGCD} and the relevant references therein, and the Minkowski content is also a popular texture index (`lacunarity') in applications characterizing the geometry of a given fractal structure beyond its fractal dimensions, see e.g. \cite{Mandelbrot94,Backes1,SSW15}.

In view of the equations \eqref{eq:deriv-vol} and \eqref{vol_deriv}, it is apparent that Minkowski contents and S-contents of a set $A$ should be closely related.
In \cite{HR12,RW13}, it was established that a bounded set $A\subset\R^d$ is Minkowski measurable if and only if it is S-measurable. More precisely, for some $0\leq D<d$, the Minkowski content $\sM^D(A)$ exists and is positive and finite  if and only if the corresponding S-content $\sS^D(A)$ exists as a positive and finite value and in this case one even has,
\begin{align}\label{eq:equal-of-contents}
\sM^D(A)=\sS^D(A).
\end{align}
The generality of this result is surprising, in view of the fact that for the corresponding lower and upper contents (i.e.\ when the existence of either the Minkowski content or the S-content is not assumed) only much weaker general relations hold: 

\begin{Theorem}\label{thm:main-rw09} \cite[Corollaries~3.2, 3.4, 3.6 and Proposition 3.7]{RW09}\\
Let $A\subset\R^d$ be bounded and $s\in[0,d]$. Then
\begin{align} \label{eq:ucont}
 \frac{d-s}{d} \usS^s(A)\leq \usM^s(A)\leq \usS^s(A),
\end{align}
where for $s=d$ the left inequality is trivial and the right inequality holds only in case $V_A(0)=0$. As a consequence, $$\udim_M A= \udim_S A\,,$$ whenever $V_A(0)=0$. Furthermore,
 \begin{align}\label{eq:lcont}
\lsS^s(A)\leq \lsM^s(A)\leq c_{d,s} \left[\lsS^{s\frac{d-1}{d}}(A)\right]^\frac{d}{d-1},
\end{align}
    where for the right hand side inequality one has to assume $d>1$ and where the constant $c_{d,s}$ just depends on the dimensions $s$ and $d$. As a consequence,
    \begin{align}\label{eq:ldim}
    \ldim_S A\le \ldim_M A \le \frac{d}{d-1} \ldim_S A.
    \end{align}
\end{Theorem}
Note that there is a fundamental difference between upper and lower contents. While the upper contents differ at most by a positive constant implying in particular the equivalence of the upper dimensions, the lower Minkowski content is in general only bounded from above by an S-content of some different dimension. This allows different lower dimensions. It was shown in \cite{W11-1}, that there exist indeed sets for which lower Minkowski dimension and lower S-dimension are different, the constants given in \eqref{eq:ldim} were even shown to be optimal, see also \cite{KLV} for the equality cases in these inequalities.

Based on the localisations of \eqref{eq:deriv-vol} and \eqref{vol_deriv}, we will discuss below local versions of the relations in \eqref{eq:equal-of-contents} and Theorem~\ref{thm:main-rw09} in the sense of relative contents. Relative Minkowski contents of a set $A\subset\R^d$ with respect to a second set $\Omega\subset\R^d$ have been studied and used e.g.\ in \cite{Zu05, LRZbook} and are defined by
\begin{align} \label{eq:rel-M-cont2}
  \sM^s(A,\Omega):=\lim_{r\se 0}\frac{\lambda_d(A_{r}\cap \Omega)}{\kappa_{d-s} r^{d-s}}
\end{align}
for $0\leq s\leq d$. We show for instance in Theorem~\ref{thm:relMmeas} below that the $s$-dimensional relative Minkowski content $\sM^s(A,\Omega)$ relative to $\Omega$ exists if and only if the corresponding relative S-content $\sS^s(A,\Omega)$ exists, provided the set $\Omega$ is chosen in a suitable (metrically associated) way.

In \cite{Winter08}, a localization of the Minkowski content (and the S-content) of a compact set $A\subset\R^d$ to a measure on this set has been suggested. The \emph{($s$-dimensional) local Minkowski content} $\mu^s(A,\cdot)$ of $A$ is defined as the weak limit (as $r\searrow 0$) of the following family of measures (whenever it exists):
\begin{align} \label{eq:old_loc_vol}
\mu_r^s(A,\cdot)&:=\frac{\lambda_d(A_r\cap\cdot)}{\kappa_{d-D} r^{d-D}}, \quad r>0.
\end{align}
Note that (in case it exists) the local Minkowski content $\mu^s(A,\cdot)$ of $A$ is a measure on $\R^d$ concentrated on $A$. Its total mass is necessarily given by the $s$-dimensional Minkowski content of $A$, i.e.\ $\mu^s(A,\R^d)=\sM^s(A)$, which implies in particular that it is necessary to choose $s=D:=\dim_M A$ and to assume that the set $A$ is $D$-dimensional Minkowski measurable for the weak limit to exist and to produce a nontrivial limit measure.

It was shown in \cite{Winter08}, that for all nonlattice self-similar sets $K\subset\R^d$ satisfying the open set condition and $D:=\dim_M K$ the limit measure exists and coincides with a multiple of the $D$-dimensional Hausdorff measure $\mu_K$ on $K$, the total mass being given by the Minkowski content $\sM^D(K)$.
In \cite{Winter08} also weak limits of curvature measures $C_k(K_\eps,\cdot)$ are discussed for self-similar sets $K$ under additional assumptions (and the results are generalized in \cite{Z11,RW09,RZ12,WZ13,BZ13}). The case $k=d-1$ is the surface area measure, which is, in fact, defined for all parallel sets of any set $A\subset\R^d$ by $C_{d-1}(A_\eps,\cdot):=\frac 12 \Ha^{d-1}(\bd A_r\cap  \cdot)$. The weak limit $\sigma^s(A,\cdot)$ of the appropriately rescaled surface measures
\begin{align} \label{eq:old_loc_surf}
\sigma_r^s(A,\cdot):=\frac{\Ha^{d-1}(\bd (A_r)\cap \cdot)}{\kappa_{d-s} (d-s) r^{d-s-1}}, r>0
\end{align}
as $r\searrow 0$, provided it exists, is regarded as the ($s$-dimensional) \emph{local S-content} in analogy with the local Minkowski content. The local S-content was shown in \cite{RW09} to exist for $s=D$ for all nonlattice self-similar sets satisfying OSC and, moreover, to coincide with the local Minkowski content. A similar relation has been observed in \cite{KK12} for nonlattice self-conformal sets in $\R$ and in \cite{KK15} for nonlattice limit sets of conformal graph directed systems in $\R$.

The natural question arises, whether there exists a general relation between the local Minkowski content of an arbitrary set and the corresponding local S-content (in case they exist)? Is the equivalence of these measures for self-similar and certain self-conformal sets just a coincidence due to self-similarity? Is it just a consequence of the fact that both measures happen to coincide with some multiple of the natural measure on these sets? Or is there a more general relation in the background? Can the general global relations between Minkowski contents and S-contents obtained in \cite{RW09,RW13} be localized in the sense of measures?

In Section~4 we give an affirmative answer to this last question, which makes clear, that the observed coincidence of the local contents in the self-similar case is not due to the self-similarity but a fundamental general relation. We will show that the local Minkowski content $\mu^D(A,\cdot)$ of an arbitrary closed set $A$ exists if and only if the corresponding local S-content $\sigma^D(A,\cdot)$ exists, and both measures coincide, see Theorems~\ref{thm:local-main} and \ref{thm:local-main3}. This holds even in the case when $A$ is unbounded, provided the weak convergence is replaced by vague convergence. We will also discuss the general properties of these local contents, in particular we will make precise the idea that these measures are locally determined, see Proposition~\ref{thm:local-main2}.

It is well known that for lattice self-similar sets averaging improves the convergence behaviour. The average Minkowski content (compare \eqref{eq:av-rel-MC}) for instance is known to exist for any self-similar set satisfying OSC, cf.~\cite{Gatzouras2000}, similarly the average S-content (compare \eqref{eq:av-rel-SC}) exists for any such set. For self-conformal sets much more is known about the existence of average contents than about the non-averaged counterparts, see e.g.~\cite{bohl13}.
In Section 5, we will therefore generalize the relations obtained in Sections 3 and 4 for relative contents and local contents to average (relative and local) contents and show for instance, that the existence of the average local Minkowski content is equivalent to the existence of the average local S-content, see Theorem~\ref{thm:local-average}. 
In Section 6 we will briefly discuss some applications. In particular, we will demonstrate that several results obtained for (local and average) S-contents in the literature can now easily be recovered from the corresponding known results for Minkowski contents. In the case of self-conformal sets, we will even derive some new results for S-contents from the existing results for Minkowski contents.

Our results clarify in which way one should choose the second set $\Omega$ in relative fractal drums $(A,\Omega)$ (as studied e.g.\ in \cite{LRZbook}), if the aim is to learn something about the geometry of the primary set $A$, namely in a metrically associated way. The examples discussed e.g.\ in \cite{LRZbook} show that all kind of strange things can happen with relative Minkowski contents even for very simple sets $A$, if the second set $\Omega$ is chosen in a too fancy fashion. However, our results show that sets $\Omega$ metrically associated with $A$ are on the one hand well behaved and suffice on the other hand to characterize the (parallel set related) geometry of $A$ completely, see also Remark~\ref{rem:piAB-suffices}.

\section{Local parallel volume and local surface area} \label{sec:two-side}

Let $A$ and $X$ be subsets of $\R^d$. Following Stacho \cite[p.370]{Stacho}, we say that $X$ is \emph{metrically associated with} $A$, if for any point $x\in X$ there exists a point $a\in \overline{A}$ so that
$d(x,a)=d(x,A)$ and all inner points of the line segment $[x,a]$ joining $x$ with $a$ belong to $X$. We write $]x,a[$ for the line segment excluding the endpoints.

First observe that the parallel sets $A_r$ of $A$ are metrically associated with $A$. Moreover, any set of the form $\pi_A^{-1}(B)$ with $B\subset A$ (or $B\subset\R^d$) is metrically associated with $A$. Here $\pi_A$ denotes the metric projection onto the set $\cl{A}$\label{page:metric-proj}. (It is defined on the set $\Unp(A)\subset\R^d$ of those points which have a unique nearest point in $\cl{A}$. Consequently, $\piAB$ is a subset of $\Unp(A)$ for any set $B$. Recall that the set $\exo(A):=\R^d\setminus\Unp(A)$, the \emph{exoskeleton} of $A$, consists of all points that do not have a unique nearest point in $\overline{A}$. Note that $\overline{A}\subset\Unp(A)$ and thus $\exo A\subset (\overline{A})^c$.)

Combining these two constructions, it is easy to see that all sets of the form $A_r\cap \piAB$ are metrically associated with $A$.

\begin{Remark}
This construction can be refined as follows:
Let $N(A)\subset\R^d\times S^{d-1}$ denote the (generalized) normal bundle of $A$, and let $\Pi_A:\Unp(A)\setminus \cl{A}\to N(A)$ be defined by $x\mapsto (\pi_A(x), \frac{x-\pi_A(x)}{|x-\pi_A(x)|})$. Then for any subset $\beta\subset N(A)$, the sets $A_r\cap \Pi_A^{-1}(\beta)$ are again metrically associated with $A$. Taking not only the base point of the metric projection into account but also its direction is a true refinement. It allows e.g.\ to study also the one sided parallel sets of a line segment or a curve.
 To keep things simple, we will formulate all results for preimages $\piAB$ of sets $B$ in the space component. However, it might be useful to keep in mind that all the results in the sequel can be extended to preimages $\Pi_A^{-1}(\beta)$ of sets $\beta$ in the normal bundle.

Note also that we will formulate all results in the sequel for closed sets $A\subset\R^d$ which seems the natural setting for our considerations. However, everything can be adapted to work for arbitrary sets $A$, if the metric projection of $A$ is defined (as above) to send points to the closure $\overline{A}$ of $A$. Since all the results will then be equivalent for a set $A$ and its closure, we can as well restrict to closed sets.
\end{Remark}

Note that any union of sets metrically associated with $A$ is again metrically associated with $A$. Stach\'o \cite{Stacho} also claimed that intersections of sets metrically associated with $A$ are metrically associated with $A$. This is not true in general, not even for two sets, as the following example shows:

\begin{example}
  Let $a_1, a_2\in\R^d$ and $A=\{a_1, a_2\}$. Let $x$ be the midpoint of the line segment $[a_1,a_2]$ and let $B_i:=[x,a_i]$, $i=1,2$. Then $B_1$ and $B_2$ are metrically associated with $A$ but $B_1\cap B_2=\{x\}$ is not.
\end{example}

The intersection stability is true, however, if attention is restricted to intersections of open sets.

\begin{Lemma}\label{lem:Kneser-local}
Let $A\subset\R^d$ and let ${\cal B}$ be a family of open subsets of $\R^d$ metrically associated with $A$. Then the intersection $\bigcap_{B\in{\cal B}} B$ is again metrically associated with $A$.
\end{Lemma}
\begin{proof}
  Let $x\in X:=\bigcap_{B\in{\cal B}} B$ and let $N_x\subset \overline{A}$ be the set of points $a$ for which $d(x,a)=d(x,A)$. Let $y\in N_x$.
  We claim that for each $B\in{\cal B}$ the line segment $]x,y[$ is contained in $B$. Obviously this implies $]x,y[\subset X$ (and $d(x,y)=d(x,A)$). Recalling that $x\in X$ was arbitrary, we conclude that $X$ is metrically
  associated with $A$ as asserted in the lemma.

  To prove the claim, let $B\in{\cal B}$ and let $(z_n)_{ n\in\N}$ be a sequence of points in $]x,y[$ converging to $x$. Since $B$ is open and $x\in B$, there is an index $n_0\in\N$ such that $z_n\in B$ for $n\geq n_0$.
  Since $B$ is metrically associated with $A$, for each $n\geq n_0$ there exists a point $a_n\in \overline{A}$ such that $d(z_n,a_n)=d(z_n,A)$ and $]z_n, a_n[\subset B$. The proof is complete, if we show that $a_n=y$ for all $n\geq n_0$, since this implies $]z_n,y[\subset B$ and thus $]x,y[\subset B$.
  Suppose that $a_n\neq y$ for some $n\geq n_0$. Then $d(z_n,a_n)\leq d(z_n,y)$. If $a_n\in]x,y[$, we obviously have $d(x,a_n)<d(x,y)$. If not, we have the same inequality, since
  $$
  d(x,a_n)<d(x,z_n)+d(z_n,a_n)\leq d(x,z_n)+d(z_n,y)=d(x,y).
  $$
  But this is a contradiction to $y\in N_x$. Hence $a_n=y$ for each $n\geq n_0$, showing that $]x,y[\subset B$ as claimed. This completes the proof of the lemma.
\end{proof}

A similar statement holds if intersections of sets are considered which are (metrically associated with $A$ and) subsets of ${\rm Unp}(A)$.

Recall that a function $f:(0,\infty)\to(0,\infty)$ is said to satisfy the \emph{Kneser property} (also called a \emph{Kneser function}) if for some integer $d$ and all $\lambda\geq 1$, the inequality
\begin{align}
  \label{eq:kneser}
  f(\lambda b)-f(\lambda a)\leq \lambda^d (f(b)-f(a)) .
\end{align}
holds for all $0<a\leq b$.
The following important observation is due to Stach\'o \cite{Stacho}:

\begin{Lemma} \label{lem:ma-implies-kneser}
  Let $A\subset\R^d$ be closed and let $X\subset\R^d$ be measurable and metrically associated with $A$.  If there exists some $t>0$ such that $\lambda_d(A_t\cap X)$ is finite, then the function $f(t):=\lambda_d(A_t\cap X), t>0$ assumes finite values for all $t>0$ and has the Kneser property.
\end{Lemma}

Stach\'o worked in fact with bounded sets $A$ in his paper and he did not give a proof but remarked that the proof of Kneser (that the global parallel volume $t\mapsto\lambda_d(A_t)$ has the Kneser property) applies also to the above $f$, cf. \cite[Hilfssatz~7, p.250]{Kneser}. Indeed, given $\lambda\geq 1$, Kneser constructed a Lipschitz mapping $g$ (with Lipschitz constant $\lambda$) from a subset of the set $A_b\setminus A_a$ onto the set $A_{\lambda b}\setminus A_{\lambda a}$ which implies the volume estimate $\lambda_d(A_{\lambda b}\setminus A_{\lambda a})=\lambda_d(g(A_{b}\setminus A_{a}))\leq\lambda^d \lambda_d(A_{b}\setminus A_{a})$ and thus the Kneser property. In the situation of Lemma~\ref{lem:ma-implies-kneser}, exactly the same pointwise construction can be used restricted to the set $X$. Therefore, the statement remains true even for unbounded sets $A$.  Indeed, if $\lambda_d(A_t\cap X)$ is finite for some $t>0$, then, by monotonicity, $f(s)$ is finite for each $0<s<t$ and, from the existence of the Lipschitz mapping $g$ restricted to $X$, the finiteness of $f(s)$ for each $s\geq t$ can be inferred together with the Kneser property.  In particular, we conclude that either $\lambda_d(A_t\cap X)=\infty$ for all $t>0$ (in which case the Kneser property makes no sense) or $\lambda_d(A_t\cap X)$ is finite for all $t>0$ (in which case the Kneser property holds). Observe that for a bounded set $A$, $\lambda_d(A_t\cap X)$ is finite for any measurable set $X$. Another sufficient (but not necessary) condition for this finiteness is that $A\cap \pi_A(X)$ is bounded, or more specifically, that $A\cap B$ is bounded and $X=\piAB$.

As mentioned in the introduction, Stach\'o related the derivative of the volume function $V_A$ of a bounded set $A\subset\R^d$ to the $(d-1)$-dimensional Minkowski content of its parallel boundaries: 
\begin{align*} 
    \sM^{d-1}(\partial (A_{r}))=\frac 12\left((V_A)_{-}'(r)+(V_A)_{+}'(r)\right),
  \end{align*}
 cf.~\eqref{eq:deriv-vol}. 
 Note that this equation does not hold in general with the open parallel set $A_r$ on the left hand side replaced by the closed parallel set $\{ z\in\R^d:\, d(z,A)\leq r\}$.

Our first aim in this section is a localization of this result. For any closed set $A\subset\R^d$ and any measurable set $B\subset\R^d$, let the function $V_{A,B}:(0,\infty)\to(0,\infty]$ be defined by
$$
V_{A,B}(r):=\lambda_d(A_r\cap \piAB), r>0.
$$
$V_{A,B}$ is the \emph{local parallel volume of $A$ relative to the set $B$}. Note that, by Lemma~\ref{lem:ma-implies-kneser}, $V_{A,B}$ is a Kneser function (provided $V_{A,B}(r)$ is finite or some and thus all $r>0$), since the set $\piAB$ is metrically associated with $A$. In order to formulate a local counterpart of \eqref{eq:deriv-vol} for $V_{A,B}$, we also need a localization of the "boundary measure" on the left hand side of this equation.
It turns out that the relevant notion of "boundary measure" in this context are relative Minkowski contents, as introduced and used e.g.\ in \cite{Zu05, LRZbook}. 
For sets $C, \Omega\subset\R^d$ and $s>0$, let
\begin{align} \label{eq:rel-M-cont}
  \sM^s(C,\Omega):=\lim_{r\se 0}\frac{\lambda_d(C_{r}\cap \Omega)}{\kappa_{d-s} r^{d-s}}
\end{align}
be the \emph{$s$-dimensional Minkowski content of $C$ relative to $\Omega$}, whenever this limit makes sense. For the moment, we will be interested in the special case when $C=\bd (A_r)$ is the parallel boundary of some set $A$ for some $r>0$, $s=d-1$ and $\Omega=\piAB$ is the preimage of some set $B$ under the metric projection onto $A$.
\begin{Theorem}[Localization of Stach\'o's Theorem] \label{thm:local-deriv-volume1}
  Let $A\subset\R^d$ be closed and let $B\subseteq\R^d$ be measurable and such that $V_{A,B}(r)$ is finite for some and hence all $r>0$.
  Then, for each $r>0$,
  \begin{align} \label{eq:local-deriv-vol}
    \sM^{d-1}(\partial(A_{r}), \piAB )=\frac 12\left((V_{A,B})_{-}'(r)+(V_{A,B})_{+}'(r)\right).
  \end{align}
\end{Theorem}
\begin{proof}
  First observe that the set $\pi_A^{-1}(B)$ is measurable and metrically associated with $A$. Hence, by Lemma~\ref{lem:ma-implies-kneser}, $V_{A,B}$ is a Kneser function.
  Therefore, left and right derivatives of $V_{A,B}$ exist for any $r>0$ and so the right hand side of \eqref{eq:local-deriv-vol} is well defined.
  From now on, we can follow the line of proof of \cite[Theorem 2]{Stacho}. By a scaling argument, it is enough to prove the equation \eqref{eq:local-deriv-vol} for $r=1$.
That is, we will show that
  \begin{align}
      \sM^{d-1}(\partial(A_{1}), \piAB))=\frac 12((V_{A,B})_{-}'(1)+(V_{A,B})_{+}'(1)).
  \end{align}
It has been observed in the proof of \cite[Theorem 2]{Stacho}, that
$$
(\partial A_{1})_t =\left(A_{(t+1)}\setminus \cl{A_{(1-t)}}\right)\setminus Y(t)
$$
where
$$
Y(t):=\left\{x\in\R^d\mid 1-t<d(x,A)<1 \text{ and } d(x,\partial(A_1))>t\right\},
$$
for which one has (in case of a bounded set $A$)
$\lim_{t\se 0} t^{-1} \lambda_d(Y(t))=0.$
Inspection of Stach\'o's argument, however, clarifies that in the unbounded case the same holds for $Y(t)$ intersected with $\piAB$, i.e., whenever $V_{A,B}(r)$ is finite for some (and thus all) $r>0$, we have
\begin{align} \label{eq:Yt}
\lim_{t\se 0} t^{-1} \lambda_d(Y(t)\cap \piAB)=0.
\end{align}
Now we can compute directly the relative Minkowski content of the boundary $\partial(A_1)$:
\begin{align*}
  \sM^{d-1}&(\partial(A_{1}), \pi_A^{-1}(B))=\lim_{t\se 0}\frac 1{2t}{\lambda_d((\partial A_{1})_t\cap \pi_A^{-1}(B))}\\
  &=\lim_{t\se 0}\frac 1{2t}{\lambda_d\left(\left(\left(A_{1+t}\setminus \cl{A_{1-t}}\right)\setminus Y(t)\right)\cap \pi_A^{-1}(B)\right)}\\
  &=\lim_{t\se 0}\left(\frac{\lambda_d\left(\left(A_{1+t}\setminus \cl{A_{1-t}}\right)\cap \pi_A^{-1}(B)\right)}{2t}-\frac{\lambda_d(Y(t)\cap \pi_A^{-1}(B))}{2t}\right).
\end{align*}
By \eqref{eq:Yt}, the second term converges to zero. The numerator of the first term can be written as
$$
\lambda_d\left(A_{1+t}\cap \pi_A^{-1}(B)\right)-\lambda_d\left(A_{1-t}\cap \pi_A^{-1}(B)\right)= V_{A,B}(1+t)-V_{A,B}(1-t)
$$
Inserting $V_{A,B}(1)- V_{A,B}(1)$, we immediately get
\begin{align*}
 \sM^{d-1}(\partial (A_{1}),& \pi_A^{-1}(B))\\
 &=\frac 12 \left(\lim_{t\se 0} \frac{V_{A,B}(1+t)-V_{A,B}(1)}{t}+\lim_{t\nearrow 0} \frac{V_{A,B}(1+t)-V_{A,B}(1)}{t}\right)\\
 &=\frac 12 \left((V_{A,B})'_+(1)+(V_{A,B})'_-(1)\right)
\end{align*}
as claimed.
\end{proof}

As a consequence of Theorem~\ref{thm:local-deriv-volume1}, we note that the relative Minkowski content $\sM^{d-1}(\partial (A_{r}), \pi_A^{-1}(B))$ is well defined for any measurable set $B$ such that $V_{A,B}$ is finite and any $r>0$. In particular, the limit involved in its definition exists as a finite value. For bounded $A$, the finiteness can also be seen from the fact that $\Omega\mapsto \sM^{d-1}(\partial A_{r}, \Omega), \Omega\subseteq\R^d$ is a monotone set function, and so in particular
$$
\sM^{d-1}(\partial (A_{r}), \pi_A^{-1}(B))\leq \sM^{d-1}(\partial (A_{r}), \R^d)=\sM^{d-1}(\partial (A_{r})).
$$
The latter Minkowski content exists and is finite, since the set $\bd (A_r)$ is $(d-1)$-rectifiable for each $r>0$, cf.~\cite[Proposition~2.3]{RW09}.



In \eqref{eq:deriv-vol}, i.e., when the volume function of the full parallel sets is considered, the Minkowski content can be replaced by the Hausdorff measure, which is due to the fact that $\bd A_r$ is $(d-1)$-rectifiable for each $r>0$. Locally this is not true in general, i.e., the relative Minkowski content of the boundary relative to some set cannot be replaced by the Hausdorff measure of the boundary restricted to this set. To illustrate this, we provide an example.
\begin{example}
  Let $A\subset\R^2$ be the union of two squares of side length 2 and distance two from each other, e.g.\ $A:=([-3,-1]\cup[1,3])\times[-1,1]$.  Consider the open vertical segment $B=](-1,-1),(-1,1)[$. Then, for $0<t<1$, $(\bd A_{1})_t\cap\pi_A^{-1}(B)$ is the interior of the rectangle $R(t)$ given by the vertices $(0,-1),(0,1), (-t,-1)$ and $(-t,1)$. Since $\lambda_2(R(t))=2t$, we have
  $$\sM^{d-1}(\bd A_{1},\pi_A^{-1}(B))=1,$$  while $$\sM^{d-1}(\bd A_{1}\cap \pi_A^{-1}(B))=\Ha^{d-1}(\bd A_{1}\cap \pi_A^{-1}(B))=0$$
  and
  $$\sM^{d-1}(\bd A_{1}\cap \cl{\pi_A^{-1}(B)})=\Ha^{d-1}(\bd A_{1}\cap \cl{\pi_A^{-1}(B)})=2.$$
\end{example}

The example also shows that, in general, the relative Minkowski content cannot easily be replaced by the usual Minkowski content of the intersection and that there is also no easy way out by considering the closure of the restricting set as one may expect at first glance. The relative Minkowski content is just the right notion to take care of ``one-sided'' situations as in the example and to derive a local relation of the form \eqref{eq:local-deriv-vol} which holds for any $r>0$.  However, for most radii $r>0$ either of the three notions can be used.
We have the following relation which generalizes \cite[Corollary 2.5]{RW09}:
\begin{Theorem} \label{thm:loc-diff-points}
  Let $A \subset\R^d$ be closed and $B \subset\R^d$ a Borel set such that $A\cap B$ is bounded. 
  Then, for all $r>0$ up to a countable set of exceptions, the function $V_{A,B}$ is differentiable at $r$ with
  \begin{align} \label{eq:local-deriv-vol2}
  V'_{A,B}(r)=\sM^{d-1}(\partial(A_{r}), \piAB )&=\Ha^{d-1} (\partial(A_{r})\cap \piAB)\\\notag &=\sM^{d-1}(\partial(A_{r})\cap \piAB ).
\end{align}
\end{Theorem}

While the first equality in \eqref{eq:local-deriv-vol2} follows from Theorem~\ref{thm:local-deriv-volume1} and the third equality is due to the $(d-1)$-rectifiability of $\bd (A_r)$, the proof of second equality requires some work. Recall that the
positive boundary $\partial^+ Z$ of a set $Z\subset\R^d$ is the set of all boundary points $z\in\bd Z$ such that there exists a point $y\notin \overline{Z}$ with $|y-z|=d(y,Z)$. Our first observation is that preimages under the metric projection only see the positive boundary of the parallel sets.
\begin{Lemma} \label{lem:bd=pos-bd}
Let $A, B \subseteq\R^d$. Then, for any $r>0$,
\begin{align} \label{eq:bd=pos-bd}
  \partial (A_r)\cap\piAB=\partial^+ (A_r)\cap\piAB.
\end{align}
In particular, $\partial (A_r)\cap\Unp(A)=\partial^+ (A_r)\cap\Unp(A)$ and therefore $\partial (A_r)\setminus \partial^+ (A_r)\subset \exo A$. More precisely, it holds $\partial^+ (A_r)\subset \Unp(A)$ and thus $\partial^+ (A_r)=\partial (A_r)\cap\Unp(A)$.
\end{Lemma}
\begin{proof}
  Let $x\in\partial (A_r)\cap\piAB$. Then $z:=\pi_A(x)\in B$ and the preimage $\pi_A^{-1}(z)$ contains a segment $[z,y]$ such that $x\in[z,y)$. We claim that the point $y$ satisfies $|y-x|=d(y,A_r)$, which implies that $x\in\bd^+ (A_r)$. Indeed, assume $y$ does not satisfy the above equation. Then there is a point $x'\in\bd (A_r)$ such that $|y-x'|<|y-x|$ and a point $z'\in A$ such that $|x'-z'|=r$. Obviously, $z'\neq z$, since otherwise
  $$
  |y-z|=|y-x|+|x-z|>|y-x'|+r=|y-x'|+|x'-z|\geq|y-z|,
  $$
  which is not possible. Moreover, we have
  $$
  d(y,A)\leq|y-z'|\leq|y-x'|+|x'-z'|<|y-x|+r\leq |y-z|,
  $$
  which implies that $y$ is not in the preimage $\pi_A^{-1}(z)$ of $z$, a contradiction. Therefore, we conclude $x\in\bd^+ (A_r)$, which shows that $\partial (A_r)\cap\piAB\subset \partial^+ (A_r)\cap\piAB$.
   The reverse inclusion is obvious. The second assertion follows from applying \eqref{eq:bd=pos-bd} to the set $B:=A$.

   For a proof of the inclusion $\partial^+ (A_r)\subset \Unp(A)$, let $z\in\partial^+ (A_r)$. Then there exists a point $y$ outside $\overline{A_r}$ such that $|y-z|=d(z,A_r)$.
   By way of contradiction, assume that $z$ does not admit a unique metric projection onto $A$, i.e. there exists at least two distinct points $a,b\in A$ such that $|z-a|=|z-b|=d(z,A)=r$. At least one of those two points, say $a$, is not on the ray from $y$ trough $z$. Therefore, we have $$|y-a|<|y-z|+|z-a|=|y-z|+r.$$
   Let $z'$ be the point on the segment $[y,a]$ such that $|z'-a|=r$. Then, in particular, $z'\in \overline{A_r}$ and $|y-a|=|y-z'|+|z'-a|=|y-z'|+r$. Plugging this into the above inequality and subtracting $r$ yields $|y-z'|<|y-z|$ and thus $d(y,A_r)<|y-z|$, which is a contradiction to the choice of $y$. Hence, $z\in\Unp(A)$.
\end{proof}

It is shown in \cite[Corollary 4.6]{HLW04}, that the global relation $(V_{A})_{+}'(r)=\Ha^{d-1}(\partial^+ A_r)$ holds for each $r>0$ for any compact set $A\subset\R^d$. Using the results in \cite{HLW04}, we derive the following local counterpart of this equation.
\begin{Proposition} \label{lem:pos-bd}
Let $A \subset\R^d$ be closed and let $B\subset\R^d$ be a Borel set such that $A\cap B$ is bounded. Then, for any $r>0$,
\begin{align} \label{eq:pos-bd}
  (V_{A,B})_{+}'(r)=\Ha^{d-1}(\partial^+ (A_r)\cap\piAB)=\Ha^{d-1}(\partial (A_r)\cap\piAB).
\end{align}
\end{Proposition}
\begin{proof}
Fix $r>0$. The second equality in \eqref{eq:pos-bd} is obvious from Lemma~\ref{lem:bd=pos-bd}. For a proof of the first equality, let $f(x):=\ind{A_r\cap\piAB}(x)$, $x\in \R^d$. Note that, by the assumed boundedness of $A\cap B$, $f$ has compact support and so, by the general Steiner formula \cite[eq.~(2.3)]{HLW04}, we have
\begin{align*}
  V_{A,B}(r)&-V_{A,B}(0)\\
  &=\sum_{i=0}^{d-1}\omega_{d-i}\int_0^\infty \int_{N(A)} t^{d-1-i} \ind{}\{t<\delta(A,x,u)\} f(x+tu)\mu_i(A, d(x,u))dt\\
  &=\sum_{i=0}^{d-1} \omega_{d-i}\int_0^r t^{d-1-i} \int_{N(A)\cap(B\times S^{d-1})} \ind{}\{t<\delta(A,x,u)\} \mu_i(A, d(x,u))dt,
\end{align*}
where $\mu_i(A,\cdot)$, $i=0,\ldots,d-1$ are the support measures of $A$, $N(A)$ is the generalized normal bundle of $A$, $\delta(A,\cdot)$ is the reach function of $A$, and $\omega_k:=\kappa_k/k$ is the $(k-1)$-dim.\ Hausdorff measure of the sphere $S^{k-1}$, cf.~\cite[Section 2]{HLW04} for more details.
Since the inner integral in each summand in this last expression is right continuous in $t$, we obtain \begin{align*}
  (V_{A,B})'_+(r)&=\sum_{i=0}^{d-1}\omega_{d-i} r^{d-1-i} \int_{N(A)\cap(B\times S^{d-1})} \ind{}\{t<\delta(A,x,u)\} \mu_i(A, d(x,u)).
\end{align*}
By \cite[Corollary 4.4]{HLW04}, the right hand side of this equation equals $2 \mu_{d-1}(A_r,\piAB\times S^{d-1})$. Using now \cite[Proposition 4.1]{HLW04} together with the fact that $\bd^{++} A_r=\bd^+ A_r$ and $\bd^2 A_r=\emptyset$ (where $\bd^{++} A_r$ consists of those boundary points $x\in\bd^{+} A_r$ for which the normal cone is $1$-dimensional and $\bd^2 A_r$ of those $x\in\bd^{++} A_r$, at which $A_r$ has two unit normals), we obtain
\begin{align*}
  (V_{A,B})'_+(r)&=2 \mu_{d-1}(A_r,\piAB\times S^{d-1})\\
  &=\int_{\bd^+ A_r} \ind{\piAB}(x)\Ha^{d-1}(dx)=\Ha^{d-1}(\bd^+ A_r\cap\piAB),
\end{align*}
which completes the proof of Proposition~\ref{lem:pos-bd}.
\end{proof}

Now we are ready to prove Theorem~\ref{thm:loc-diff-points}.
\begin{proof}[Proof of Theorem~\ref{thm:loc-diff-points}.]
The differentiability of $V_{A,B}$ for all $r>0$ except countably many is clear from the Kneser property of $V_{A,B}$, cf. Lemma~\ref{lem:ma-implies-kneser}. So let $r>0$ be a differentiability point of $V_{A,B}$. Then the first equality in \eqref{eq:local-deriv-vol2} is obvious from Theorem~\ref{thm:local-deriv-volume1} and the second equality follows from Proposition~\ref{lem:pos-bd}. Moreover, the third equality follows from the fact (proved e.g.\ in \cite[Proposition 2.3]{RW09}) that $\partial(A_{r})$ (and thus any of its subsets) is $(d-1)$-rectifiable and the equality of Minkowski content and Hausdorff measure for such sets, see e.g.~\cite[Section 3.2.39]{F69}.
%
%
\end{proof}

In an earlier version of the proof of Theorem~\ref{thm:loc-diff-points}, we have used the fact that the special relative Minkowski contents we use above are actually measures in the argument $B$. For differentiability points $r$, this can now be deduced from Theorem~\ref{thm:loc-diff-points}, since the Hausdorff measure is obviously a measure. (Note that the proof above does not use this observation.) In general, this is not obvious at all, but it follows from the properties of Kneser functions, as the next result shows.
\begin{Lemma} \label{lem:rel-Mink-is-measure}
For any compact set $A\subset\R^d$ and any $r>0$, the set function $B\mapsto \sM^{d-1}(\partial(A_{r}), \piAB ), B\in\sB(\R^d)$ is a finite measure.
\end{Lemma}
\begin{proof}
Fix $A\subset\R^d$ and $r>0$. We show that $\sM^{d-1}(\partial(A_{r}),\pi_A^{-1}(\cdot))$ is $\sigma$-additive. It suffices to show that for any sequence of $B_1, B_2, \ldots$ of pairwise disjoint Borel sets, we have
\begin{align*}
  \sM^{d-1}(\bd (A_r),\pi_A^{-1}(\bigcup_{j=1}^\infty B_j))=\sum_{i=1}^\infty \sM^{d-1}(\bd (A_r),\pi_A^{-1}(B_j)).
\end{align*}
First observe that, by Lemma~\ref{lem:ma-implies-kneser}, 
for each $j\in \N$, $V_{A,B_j}$ is a Kneser function and so is $V_{A,B}$ for $B:=\bigcup_{j=1}^\infty B_j$. Moreover, the sum $\sum_j V_{A,B_j}$ converges pointwise to $V_{A,B}$, which follows from the $\sigma$-additivity of the Lebesgue measure and the fact that the sets $A_r\cap\pi^{-1}_A(B_j), j\in\N$ are pairwise disjoint and measurable.
Now, by \cite[Lemma 4]{Stacho}, we infer that also the relations $\sum_j (V_{A,B_j})'_+(r)=(V_{A,B})'_+(r)$ and $\sum_j (V_{A,B_j})'_-(r)=(V_{A,B})'_-(r)$ hold for each $r>0$. Combining this with Theorem~\ref{thm:local-deriv-volume1}, the above claim follows.
\end{proof}

\section{General relations between relative Minkowski and S-contents}

In analogy with relative Minkowski contents, we will now introduce and discuss a relative version of the S-content.
For $A,\Omega\subset\R^d$ and $0\leq s<d$, we define
\begin{align}\label{eq:def-rel-S-cont}
{\cal S}^s(A,\Omega):=\lim_{r\searrow 0} \frac{{\Ha}^{d-1}(\bd A_r\cap \Omega)}{(d-s)\kappa_{d-s}r^{d-1-s}}\,,
\end{align}
whenever this limit exists and call it the \emph{$s$-dimensional S-content of $A$ relative to $\Omega$}.
(In particular, it is necessary that ${\Ha}^{d-1}(\bd A_r\cap \Omega)$ is finite for all $r>0$ sufficiently small. This finiteness is for instance ensured if $A$ or $\Omega$ are bounded or if $A\cap B$ is bounded and $\Omega=\piAB$.)
  In case the limit does not exist, we can consider upper and lower versions, $\usS^s(A,\Omega)$ and $\lsS^s(A,\Omega)$, respectively, by replacing the limit with $\limsup$ and $\liminf$.

Note that in case of a bounded set $A$ for the choice $\Omega=\R^d$ (or $\Omega=A_\eps$ for some fixed $\eps>0$), we recover the (global) S-content, $\sS^s(A,\R^d)=\sS^s(A, A_\eps)=\sS^s(A)$. It is convenient to set ${\cal S}^d(A,\Omega):=0$ for completeness, which is justified by the fact that $\lim_{r\to 0}r{\cal H}^{d-1}(\bd A_r)=0$ for any bounded set $A\subset\R^d$, cf.~\cite{Kneser}. From this observation it can easily be derived that $\lim_{r\to 0}r{\cal H}^{d-1}(\bd A_r\cap \Omega)=0$
for any bounded set $\Omega$, even when $A$ is unbounded. (Just intersect $A$ with a large enough ball containing $\Omega$.)
Therefore, it is justified to set
\begin{align}
  \label{eq:Sd}
  \sS^d(A,\Omega):=0
\end{align}
for any set $A$ and any Borel set $\Omega$.

A priori, it is not clear whether it leads to the same notion of relative S-content, if in the above definition the Hausdorff measure is replaced by a different notion of surface area, the relative $(d-1)$-dimensional Minkowski content.
Using the results of the previous section, this can be established at least if $\Omega$ is a preimage under the metric projection $\pi_A$.

\begin{Lemma}
  Let $A\subset\R^d$ be closed. Then for any Borel set $B\subset\R^d$ and $s\in[0,d]$,
\begin{align*}
  \usS^s(A,\piAB)=\limsup_{r\searrow 0} \frac{{\sM}^{d-1}(\bd A_r,\piAB)}{(d-s)\kappa_{d-s}r^{d-1-s}}
\end{align*}
and
\begin{align*}
  \lsS^s(A,\piAB)=\liminf_{r\searrow 0} \frac{{\sM}^{d-1}(\bd A_r,\piAB)}{(d-s)\kappa_{d-s}r^{d-1-s}}.
\end{align*}
\end{Lemma}

This means in particular, that if the relative S-content $\sS^s(A,\piAB)$ exists, then it is equivalently given with the Hausdorff measure in its definition replaced by the relative Minkowski content.

\begin{proof}
  Combining Proposition~\ref{lem:pos-bd} with Theorem~\ref{thm:local-deriv-volume1} and the fact that $(V_{A,B})'_+(r)\leq(V_{A,B})'_-(r)$, we have for each $r>0$,
  \begin{align*}
    {\Ha}^{d-1}(\bd A_r\cap \piAB)=(V_{A,B})'_+(r)\leq {\sM}^{d-1}(\bd A_r,\piAB),
  \end{align*}
  which implies that the $\leq$-relation holds in both of the above equations. For the reverse inequalities note that, for each $r>0$, ${\sM}^{d-1}(\bd A_r,\piAB)\leq (V_{A,B})'_-(r)$. Since the function $(V_{A,B})'_-$ is left-continuous and coincides with the derivative $(V_{A,B})'$ except on a countable set, $(V_{A,B})'_-(r)$ can be approximated from the left arbitrary well by its values at differentiability points $t>r$, at which $(V_{A,B})'(t)={\Ha}^{d-1}(\bd A_t\cap \piAB)$ holds, by Theorem~\ref{thm:loc-diff-points}. This implies the $\geq$-relation in the two equations of the statement and completes the proof.
\end{proof}



It will become clear later, that it is enough to study the special relative S-contents of the form $\sS^s(A,\piAB)$ with $B\in\sB(\R^d)$ to get a complete picture of the local behaviour of the parallel surface area of a set $A$, cf.\ Remark~\ref{rem:piAB-suffices}.
We proceed by discussing the precise relation between relative S-contents and relative Minkowski contents relative to such preimages $\piAB$.

First we will demonstrate that a certain analogue of Theorem~\ref{thm:main-rw09} holds for the relative upper and lower contents, which are defined in the obvious way by replacing the limits in \eqref{eq:rel-M-cont} and \eqref{eq:def-rel-S-cont} by limsup and liminf, respectively.

\begin{Theorem}\label{thm:u-rel-contents}
Let $A\subset\R^d$ be closed and $s\in[0,d]$. Then, for any Borel set $B$ such that $A\cap B$ is bounded,
\begin{align} \label{eq:u-rel-cont}
 \frac{d-s}{d} \usS^s(A,\piAB)\leq \usM^s(A,\piAB)\leq \usS^s(A,\piAB),
\end{align}
where the right hand inequality holds only in case $\lambda_d(A\cap B)=0$. Furthermore,
 \begin{align}\label{eq:l-rel-cont}
\lsS^s(A,\piAB)\leq \lsM^s(A,\piAB).
\end{align}
\end{Theorem}

\begin{proof}
   We start with the right hand inequality in \eqref{eq:u-rel-cont} and $s=d$.  By the hypothesis $\lambda_d(A\cap B)=0$ and the continuity of the volume function, we have
  $$
  0\leq\lambda_d(A_r\cap\pi^{-1}_A(B))=\lambda_d(A_r\cap\pi^{-1}_A(A\cap B))\leq \lambda_d((A\cap B)_r)\to 0,\quad \text{ as } r\searrow 0,
  $$
  from which we infer
     $0\leq\lsM^d(A,\piAB)\leq \usM^d(A,\piAB)\leq \usM^d(A\cap B)=0$. This shows the right hand inequality in \eqref{eq:u-rel-cont} for $s=d$. So let us now assume $s<d$.
  By Lemma~\ref{lem:Kneser-local}, the function $f(r):=V_{A,B}(r)=\lambda_d(A_r\cap\piAB), r\geq0$ is a Kneser function (with $f(0)=0$ because of the hypothesis $\lambda_d(A\cap B)=0$).
   Moreover, the function $h(r):=\kappa_{d-s} r^{d-s}$, $r\geq 0$ is differentiable with $h(0)=0$ and its derivative $h'(r)=\kappa_{d-s}(d-s)r^{d-s-1}$ is nonzero on some right neighbourhood of $0$. Hence, we can apply the general version of Proposition 3.1 in \cite{RW09}, cf.~\cite[the proposition on p.1672]{RW09} or \cite[Proposition 3.1]{RW13}, and conclude that
  \begin{align} \label{eq:u-rel-cont-proof1}
     \usM^s(A,\piAB)=\limsup_{r\to 0} \frac{f(r)-f(0)}{h(r)}\leq \limsup_{r\to 0} \frac{f'(r )}{h'(r)}=\limsup_{r\to 0} \frac{V_{A,B}'(r)}{h'(r)},
  \end{align}
where the limes superior is taken over those $r>0$ for which the derivative in the numerator exists. Note that the conclusion can be equivalently formulated with the left or right derivative, $f'_-$ and $f'_+$, of $f$, which exist for any $r>0$. This is due to the left- and right-continuity of the functions $f'_-$ and $f'_+$, respectively, cf.~\cite[Lemma 2]{Stacho}.
Now, by Proposition~\ref{lem:pos-bd}, we can replace  $f_+'(r)=(V_{A,B})_+'(r)$ by $\Ha^{d-1} (\partial(A_{r})\cap \piAB)$, which means that the last limes superior in \eqref{eq:u-rel-cont-proof1} equals $\usS^s(A,\piAB)$, showing the right hand inequality in \eqref{eq:u-rel-cont} for $s<d$.

From the second assertion in \cite[Proposition 3.1]{RW09} one obtains in a similar way the inequality in \eqref{eq:l-rel-cont}: provided $f(0)=0$ (and $s<d$), we directly get
\begin{align} \label{eq:u-rel-cont-proof2}
     \lsM^s(A,\piAB)=\liminf_{r\to 0} \frac{f(r)}{h(r)}\geq \liminf_{r\to 0} \frac{f'(r )}{h'(r)}=\lsS^s(A,\piAB).
  \end{align}
In case $f(0)>0$ and $s<d$, the left hand side in this inequality is $+\infty$ and therefore it trivially holds in this case.
In the remaining case $s=d$, \eqref{eq:l-rel-cont} holds trivially, since $0\leq\lsS^d(A,\piAB)\leq\usS^d(A,\piAB)=0$, cf. \eqref{eq:Sd}, which completes the proof of \eqref{eq:l-rel-cont}.

It remains to prove the left inequality in \eqref{eq:u-rel-cont}. Let $s<d$. (Otherwise the constant on the left is zero and the inequality holds trivially).  We apply the following special form of \cite[Proposition 2.1]{RW13} (cf.\ also the comments in its proof) to the Kneser function $f(r):=V_{A,B}(r), r\geq 0$: Let $f$ be a Kneser function of order $d\geq 1$ such that $M:=\limsup_{r\to 0} f(r)/r^{d-s}<\infty$ for some $s<d$. Then
$$
\limsup_{r\searrow 0} \frac{f'_-(r)}{(d-s)r^{d-s-1}} \leq \frac{d}{d-s} M.
$$
Recalling that $f'_+(r)\leq f'_-(r)$ for any $r>0$, the desired inequality follows again from Proposition~\ref{lem:pos-bd}.
\end{proof}

We point out that there seems to be no obvious local analogue of the last inequality in \eqref{eq:lcont} in Theorem~\ref{thm:main-rw09} providing an upper bound for the lower Minkowski content in terms of the lower S-content. The proof of this global inequality is based on the isoperimetric inequality and a localization would require a suitable local analogue of this latter result. Note that the hypothesis $\lambda_d(A\cap B)=0$ is not only sufficient for the right hand inequality in \eqref{eq:u-rel-cont} to hold but is essentially also necessary. Indeed, if $\lambda_d(A\cap B)>0$, then $\usM^s(A,\piAB)=\infty$ for any $s<d$ and so the inequality either fails (in case $\usS^s(A,\piAB)$ is finite, which happens e.g.\ for any set $A$ with $d-1$-rectifiable boundary and any $s\ge d-1$) or it has no force. For $s=d$, on the other hand, we have $\usM^d(A,\piAB)=\sM^d(A,\piAB)=\lambda_d(A\cap B)>0$ while $\usS^d(A,\piAB)=0$, cf.~\eqref{eq:Sd}, and so the inequality always fails.

As a useful consequence of the above result, we note the following.
\begin{Corollary}\label{cor:rel-contents-zero}
Let $A\subset\R^d$ be closed and $s\in[0,d)$. Then, for any Borel set $B$ such that $\lambda_d(A\cap B)=0$,
\begin{align*} 
 \sM^s(A,\piAB)=0 \quad\text{ if and only if }\quad \sS^s(A,\piAB)=0.
\end{align*}
More specifically, if $A\subset\R^d$ is a compact set with $\udim_M A<d$,
then this equivalence holds for any Borel set $B\subset\R^d$.
\end{Corollary}
\begin{proof}
  Assume that $\sM^s(A,\piAB)=0$. Then, in particular $\usM^s(A,\piAB)=0$, which, by the left inequality in \eqref{eq:u-rel-cont} of Theorem~\ref{thm:u-rel-contents}, implies $\usS^s(A,\piAB)=0$. But the latter implies that $\sS^s(A,\piAB)$ exists and equals zero, showing one direction of the assertion. The other implication follows similarly, employing the right hand inequality in \eqref{eq:u-rel-cont}.
\end{proof}

Now we are ready to provide a local version of the fundamental relation \eqref{eq:equal-of-contents}.
\begin{Theorem} \label{thm:relMmeas}
Let $A$ be a closed subset of $\R^d$. Let $B\subset \R^d$ be some Borel set, 
$D\in[0,d)$ and $M\in(0,\infty)$. Then
$$
\sM^D(A,\piAB)=M,
$$
if and only if
$$
\sS^D(A,\piAB)=M.
$$
\end{Theorem}
That is, the set $A$ is ($D$-dimensional) \emph{Minkowski measurable} relative to $\piAB$ if and only if it is ($D$-dimensional) \emph{S-measurable} relative to $\piAB$,
and in this case both relative ($D$-dimensional) contents coincide.

\begin{proof}
Assume that  $\sM^D(A,\piAB)=M$. Observe that $f(r):=V_{A,B}(r)$ is a Kneser function. Let $s:=d-D$ and apply \cite[Proposition 2.3]{RW13}, which yields
$$
\lim_{r\searrow 0} \frac{(V_{A,B})'_+(r)}{(d-D)\kappa_{d-D} r^{d-D-1}}=\lim_{r\searrow 0} \frac{(V_{A,B})'_-(r)}{(d-D)\kappa_{d-D} r^{d-D-1}}=M.
$$
Then $
\sS^D(A,\piAB)=M$ follows from Proposition~\ref{lem:pos-bd}, which states in particular that
$(V_{A,B})'_+(r)=\Ha^{d-1}(\bd A_r\cap \piAB)$.

For a proof of the reverse implication, assume $\sS^D(A,\piAB)=M$ and apply \cite[Proposition~3.1]{RW13} to the function $f$ as above and $h(r)=\kappa_{d-D}r^{d-D}$, from which the assertion $\sM^D(A,\piAB)=M$ follows.
\end{proof}

It is worth noting that also several other results on the relation between Minkowski contents and S-contents in \cite{RW09,RW13} carry over to their relative counterparts. In particular, Theorem 2.2 in \cite{RW13} reads for the relative contents as follows:
\begin{Theorem} \label{thm:twoside}
Let $A\subset\R^d$ be closed, $B\subset\R^d$ some Borel set and $D\in[0,d)$. Then
\begin{align*}
  0<\lsM^D(A,\piAB)\le\usM^D(A,\piAB)<\infty
\end{align*}
if and only if
\begin{align*}
  0<\lsS^D(A,\piAB)\le\usS^D(A,\piAB)<\infty.
\end{align*}
\end{Theorem}
Similary, the results for generalized contents in \cite{RW13} carry over, if the relative contents are generalized
replacing the power functions $\eps^{t}$ in the definitions of the contents by more general gauge functions in the obvious way. In particular, Theorems 3.4 and 3.7 in \cite{RW13} have local analogues.
The ingredients of the proofs are always the same: Proposition~\ref{lem:pos-bd} together with the appropriate statement on Kneser functions from \cite{RW13}.

\section{Localization as measures} \label{sec:local}

The following Theorem gives an affirmative answer to this last question for compact sets. In fact, these localizations even make sense for unbounded sets $A$. Therefore, we will later discuss an extension to unbounded sets. This will require some additional considerations regarding the notion of convergence and a different technique of proof.

%
%
%
\begin{Theorem}\label{thm:local-main}
Let $A\subset\R^d$ be a compact set and let $s\in[0,d)$. 
Let $\mu$ be a finite Borel measure on $A$. Then
\begin{align} \label{eq:mu_weakconv}
  \mu_r^s(A,\cdot)\wlim\mu \text{ as } r\searrow 0
\end{align}
if and only if
\begin{align} \label{eq:sigma_weakconv}
  \sigma_r^s(A,\cdot)\wlim\mu \text{ as } r\searrow 0.
 \end{align}
In this case, the total mass $\mu(A)$ of the measure $\mu$ necessarily coincides with the Minkowski content and the S-content of $A$, i.e., $\dim_M A=s$ and
$$
\mu(A)=\sM^s(A)=\sS^s(A).
$$
\end{Theorem}

The case $s=d$ is naturally excluded. Recall that always $\sS^d(A)=0$, see \eqref{eq:Sd}, and thus the weak limit of the measures $\sigma_r^d(A,\cdot)$ as $r\searrow 0$, exists and is the null measure. At the same time the local rescaled parallel volume $\mu_r^d(A,\cdot)$ will converge to the Lebesgue measure restricted to $A$, which may either be the null measure (in case $\lambda_d(A)=0$) or not.

First we give a rough outline of the proof. We will construct a separating class of the Borel $\sigma$-algebra ${\cal B}^d$ of $\R^d$ which is adapted to the structure of the set $A$ under consideration. It will consist of preimages under the metric projection onto $A$ and sets bounded away from $A$ by some positive distance.
Since Borel measures that coincide on a separating class are necessarily the same, and since $\mu$ carries no mass outside the set $A$, this reduces the task of comparing the two limits in Theorem~\ref{thm:local-main} essentially to comparing the measure of preimages of Borel sets under the metric projection, for which the  results of the previous section can be employed.

We start by defining the separating class: For a closed set $A\subset\R^d$,
let
$$
\sC_A:=\{B\in\sB^d: d(A,B)>0)\},
$$
where $d(A,B):=\inf\{d(a,b):a\in A,b\in B\}$. Denote $\sB^d(A)$ the $\sigma$-algebra of Borel subsets of $A$ and let $\pi_A^{-1}\sB^d:=\{\pi_A^{-1}(B): B\in\sB^d(A)\}$. Since $\pi_A:\Unp(A)\to A$ is continuous, $\pi_A^{-1}\sB^d$ is a $\sigma$-algebra on $\Unp(A)$ contained in the Borel $\sigma$-algebra $\sB^d(\Unp(A))$. Let
\begin{align}
\sA_A:= \pi_A^{-1}\sB^d\cup \sC_A.
\end{align}

\begin{Proposition} \label{prop:stable}
For any closed set $A\subset\R^d$, the set family $\sA_A$ is an intersection stable generator of the Borel $\sigma$-algebra $\sB^d$.
\end{Proposition}
\begin{proof}
The stability with respect to intersections is easily seen. Let $C_1,C_2\in\sA_A$. If at least one of the two sets is in $\sC_A$, then $d(A,C_1\cap C_2)=\max\{d(A, C_1), d(A,C_2)\}>0$ and thus $C_1\cap C_2\in\sC_A$. If both sets are in $\pi_A^{-1}\sB^d$ then clearly their intersection $C_1\cap C_2$ is also an element of this $\sigma$-algebra.

Since the class $\sA_A$ consists of Borel sets, the $\sigma$-algebra $\sigma(\sA_A)$ generated by $\sA_A$ must be contained in $\sB^d$, that is, $\sigma(\sA_A)\subset\sB^d$. It remains to prove the reverse inclusion. 
Let $B\subseteq\R^d$ be some Borel set.
First observe that
\begin{align}\label{eq:stable}
A^c\cap B=\bigcup_{n\in\N} \left((A_{1/n})^c\cap B\right)
\end{align}
and that the sets in this union are elements of $\sC_A$. Indeed, if $z\in A^c\cap B$, then, since $A$ is closed, $d(z,A)>1/n$ for some sufficiently large $n$ and so $z\in (A_{1/n})^c\cap B$. The reverse inclusion is obvious from $A\subset A_{1/n}$.

Moreover, we have
$$
A\cap B =\pi_A^{-1}(A\cap B)\cap A,
$$
where the first set on the right is obviously in $\pi_A^{-1}\sB^d$ and the set $A$ is in $\sigma(\sC_A)$, since $A^c$ is. The latter is seen by choosing $B=\R^d$ in \eqref{eq:stable}.

We conclude that
\begin{align}
  B=(A\cap B) \cup (A^c\cap B)= \left(\pi_A^{-1}(A\cap B)\cap A\right)\cup  \bigcup_n \left((A_{1/n})^c\cap B\right),
\end{align}
which is a representation of $B$ as a countable union of sets in $\sigma(\sA_A)$, showing that $B\in\sigma(\sA_A)$. This completes the proof.
\end{proof}

\begin{Remark}
  Let $\sA_A':=\sA_A\cup\{\exo A\}$. Obviously, the enlarged family $\sA_A'$ still generates $\sB^d$. Moreover, $\sA_A'$ is still intersection stable, since $C\cap\exo A$ is in $\sC_A$ for any $C\in\sC_A$ and empty for any set $C\in\pi_A^{-1}\sB^d$. Hence Proposition~\ref{prop:stable} holds analogously for $\sA_A'$.
\end{Remark}
As a consequence of Proposition~\ref{prop:stable} and the uniqueness theorem for measures, we get the following statement.
\begin{Corollary} \label{cor:meas-uniqueness}
Let $A$ be a closed set in $\R^d$.
\begin{enumerate}
  \item[(i)] If two Borel measures $\mu,\nu$ on $\R^d$ satisfy $\mu(B)=\nu(B)$ for all sets $B\in\sA_A$ (or $\sA_A'$) and if there exists sets $E_n\in{\sA_A}$ (or $\sA_A'$) such that $\bigcup_n E_n=\R^d$ and $\mu(E_n)=\nu(E_n)<\infty$ for any $n\in\N$, then $\mu=\nu$.
  \item[(ii)] If two finite Borel measures $\mu,\nu$ on $E\subseteq\R^d$ whose support is contained in $A\cap E$ satisfy $\mu(C)=\nu(C)$ for all sets $C\in\pi_A^{-1}\sB^d(E)$, then $\mu=\nu$.
  \item[(iii)] If two locally finite Borel measures $\mu,\nu$ on $E\subseteq\R^d$ whose support is contained in $A\cap E$ satisfy $\mu(C)=\nu(C)$ for all sets $C\in\pi_A^{-1}\sB^d(E)$ such that $\mu(\bd C)=0$, then $\mu=\nu$.
\end{enumerate}
\end{Corollary}
\begin{proof}
(i) By Proposition~\ref{prop:stable}, $\sA_A$ is an intersection stable generator of $\sB^d$. 
Hence the assertion in (i) is just a special case of the uniqueness theorem for measures (see e.g. \cite[p.60]{els} or \cite[Satz 5.4, p.26]{Bauer}).

(ii) If $\spt\mu\subseteq A$ and $\spt \nu\subseteq A$, then $\mu(B)=0=\nu(B)$ for any set $B\in\sC_A$ and $\mu(\exo A)=\nu(\exo A)=0$. Together with the hypothesis this implies $\mu$ and $\nu$ coincide on the whole family $\sA_A'$. Furthermore, let $E_n:=\pi_A^{-1}(B_n)$ for $n\in\N$, where $B_n$ is a ball of radius $n$ centered at the origin, and let $E_0:=\exo A$. Observe that $E_n\in\sA_A'$ for all $n\in\N_0$ and hence $\mu(E_n)=\nu(E_n)$, and that $\bigcup_{n=0}^\infty E_n=\R^d$. Moreover, we have the decomposition
$
E_n=(\pi_A^{-1}(B_n)\cap A_1)\cup(\pi_A^{-1}(B_n)\cap (A_1)^c),
$
in which the second set is clearly an element of $\sC_A$, hence $\mu(\pi_A^{-1}(B_n)\cap (A_1)^c)=0$.
The first set is compact such that, by the assumed local finiteness of the two measures, $\mu(\pi_A^{-1}(B_n)\cap A_1)<\infty$.
This implies the finiteness of $\mu(E_n)$ for any $n\in\N$. Therefore, we can apply
part (i), which yields the assertion (ii).

(iii) Let the Borel measures $\mu',\nu'$ be defined by $\mu'(B):=\mu(\piAB)$ and $\nu'(B):=\nu(\piAB)$ for $B\in\sB(\R^d)$. The assumption in (iii) implies $\mu'(B)=\nu'(B)$ for all $\mu'$-continuity sets $B$. (To see this, let $B$ be a $\mu'$-continuity set, that is, assume that $\mu'(\bd B)=0$. By definition of $\mu'$, this means $\mu(\pi_A^{-1}(\bd B))=0$.  Since $\bd (\piAB) \subset \pi_A^{-1}(\bd B)\cup \exo(A)$ and $\mu(\exo(A))=0$, we infer that $\mu(\bd \piAB)=0$, i.e. $\piAB$ is a $\mu$-continuity set. Hence, by the hypothesis in (iii), we have $\mu(\piAB)=\nu(\piAB)$ and therefore $\mu'(B)=\nu'(B)$ as claimed.)

Consider the constant sequence $(\nu_i)$ of measures defined by $\nu_i:=\nu'$. Then obviously $\nu_i$ converges vaguely to $\nu'$ as $i\to\infty$. But we also have $\nu_i\vlim\mu'$ as $i\to\infty$, by the Portmanteau Theorem, since $\nu_i(B)\to\mu'(B)$ for any $\mu'$-continuity set $B\in\sB(\R^d)$. But a vague limit is unique in case it exists and thus we get $\mu'=\nu'$. For the measures $\mu$ and $\nu$ this means, we have in fact
$\mu(\piAB)=\nu(\piAB)$ for any $B\in\sB(\R^d)$ and therefore part (ii) applies and the claim follows.
\end{proof}

\begin{Remark} \label{rem:piAB-suffices}
  The assertion (ii) in Proposition~\ref{cor:meas-uniqueness} clarifies, that local Minkowski content and local S-content of a  set $A$ (which are measures on $A$, in case they exist) are completely determined by their values for sets of the form $\piAB$, that is, by the relative Minkowski contents and S-contents, respectively, relative to sets $\piAB$, $B$ Borel. For a complete understanding of the local geometry of $A$ (as far as it is related to the parallel sets), it therefore suffices to study relative contents relative to sets metrically associated with $A$ (of the special form $\piAB$).
\end{Remark}

\begin{proof}[Proof of Theorem~\ref{thm:local-main}]
Assume $ \mu_r^D(A,\cdot)\wlim\mu $ as $r\searrow 0$. Then, by the Portmanteau Theorem,
\begin{align}
  \label{eq:mu-continuity}
  \lim_{r\to 0} \mu_r^D(A,\piAB)=\mu(\piAB)
\end{align}
for all Borel sets $B\subset\R^d$ such that $\mu(\bd \piAB)=0$.

Since the measure $\sigma_r^D(A,\cdot)$ has its support contained in $A_r$, the family $\{\sigma_r^D(A,\cdot):r\in(0,1]\}$ is tight. Hence, by Prohorov's Theorem, every sequence in this family has a weakly converging subsequence. Let $(r_i)_{i\in\N}$ be a null sequence such that the measures $\sigma_{r_i}^D(A,\cdot)$ converge weakly as $i\to\infty$. Let $\nu$ be the limit measure of this sequence (which is necessarily concentrated on $A$). We will show that $\nu$ coincides with $\mu$. Since the choice of the sequence $(r_i)$ was arbitrary, the limit measure must then be the same for every such sequence, which implies the weak convergence $\nu_r^D(A,\cdot)\wlim\mu$ as $r\searrow 0$ as desired.

To show the equality $\nu=\mu$, by Corollary~\ref{cor:meas-uniqueness} (iii), it suffices to prove that $\nu(\piAB)=\mu(\piAB)$ for all Borel sets $B\subset\R^d$ such that $\mu(\bd \piAB)=0$. So let $B$ such a set. Since, by the observation \eqref{eq:mu-continuity} above, the relative Minkowski content
$$\sM^D(A,\piAB)=\lim_{r\to 0} \mu_r^D(A,\piAB)$$
exists and coincides with $\mu(\piAB)$, we can infer from Theorem~\ref{thm:relMmeas} (if $\mu(\piAB)>0$) or otherwise from Corollary~\ref{cor:rel-contents-zero}, that also the limit
$$
\sS^D(A,\piAB)=\lim_{r\to 0} \sigma_r^D(A,\piAB)=\lim_{i\to\infty} \sigma_{r_i}^D(A,\piAB)
$$
exists and equals $\mu(\piAB)$. This shows $\nu=\mu$ and  completes the proof of the implication \eqref{eq:mu_weakconv} $\Rightarrow $ \eqref{eq:sigma_weakconv} in Theorem~\ref{thm:local-main}.

The proof of the reverse implication is exactly the same with the roles of the families $\mu_r^D(A,\cdot)$ and $\sigma_r^D(A,\cdot)$ interchanged.
\end{proof}

{\bf Unbounded sets $A$.}
Now we discuss an extension of Theorem~\ref{thm:local-main} to unbounded sets. For this we need to extend the notion of local Minkowski content and local S-content to arbitrary closed sets $A\subset\R^d$. First observe that \eqref{eq:old_loc_vol} and \eqref{eq:old_loc_surf} still define measures $\mu_r^s(A,\cdot)$ and $\sigma_r^s(A,\cdot)$ for any $s\geq 0$ and $r>0$ in case of an unbounded set $A$. Note that these measures are locally finite and inner regular and thus Radon measures. For studying their limiting behaviour as $r\se 0$, we use the notion of
vague convergence of measures. 
Recall that a family  $\{\mu_r: r>0\}$ of Radon measures on some open set $E\subseteq\R^d$ (or, more generally, on some locally compact space $E$) is said to converge vaguely to a Radon measure $\mu$ on $E$ as $r\searrow 0$, if and only if $\lim_{r\se 0}\int f d\mu_r=\int f d\mu$ for any continuous real valued function $f$ on $E$ with compact support. We write $\mu_r\vlim \mu$ for this convergence. A necessary and sufficient condition for the vague convergence $\mu_r\vlim\mu$ as $r\searrow 0$ is the following: for any compact set $K\subset E$ and any open, relative compact set $G\subset E$ one has
\begin{align}
\label{eq:vag-cond}
\limsup_{r\searrow 0}\mu_r(K)\leq\mu(K) \quad \text{ and } \quad \liminf_{r\searrow 0}\mu_r(G)\geq\mu(G),
\end{align}
respectively, cf.~e.g.~\cite[p.219]{Bauer}. We will study the vague convergence of the families $\{\mu_r^s(A,\cdot):r>0\}$ and $\{\sigma_r^s(A,\cdot):r>0\}$ as $r\se 0$ for closed (unbounded) sets $A\subset\R^d$ and suitable exponents $s\in[0,d)$.

In case of a compact set $A$, the existence and finiteness of the Minkowski content $\sM^s(A)$ is a necessary condition for the weak convergence of the measures $\mu_r^s(A,\cdot)$ (and $\sigma_r^s(A,\cdot)$), which allows to easily single out the right dimension $s$ for which the local contents are interesting. In case of an unbounded set $A$, the total mass of these measures is infinite, in general, and it is not obvious what the right scaling exponent $s$ is for a given $A$. It seems that there is no generally agreed notion of Minkowski dimension (or Minkowski content) in case of unbounded sets. We suggest the following definition: Let $(B_n)_{n\in\N}$ be an increasing sequence of bounded open sets in $\R^d$ such that their union covers $\R^d$, i.e.\ $\bigcup_n B_n=\R^d$. For $A\subset\R^d$, we define the (upper) Minkowski dimension of $A$ by
\begin{align}
  \label{eq:def-mink-unbd} \udim_M A:=\inf\{s\geq 0: \usM^s(A,\pi_A^{-1}(B_n))<\infty \text{ for all } n\in\N\}
\end{align}
Note that for $A$ bounded, this definition coincides with the standard one, since for all $n$ large enough the set $A$ will be contained in $B_n$ and $\usM^s(A,\pi_A^{-1}(B_n))=\usM^s(A)$ holds. Moreover, the definition does not depend the choice of the sets $B_n$. Replacing $(B_n)$ with any other increasing sequence of bounded open sets $B'_n$ such that $\bigcup_n B'_n=\R^d$ will produce the same number for the upper Minkowski dimension. In a similar way, we can define the lower Minkowski dimension $\ldim_M A$ of $A$ by replacing the upper relative Minkowski contents in \eqref{eq:def-mink-unbd} by their lower countertparts. If both numbers coincide, the common value can be regarded as the Minkowski dimension $\dim_M A$. The (upper and lower) S-dimension of an unbounded set $A$ can be defined in a completely analogous way.

It is natural and convenient to use relative contents relative to preimages $\pi_A^{-1}(B_n)$ under the metric projection onto $A$ for the definition. Note that, in general, the sets $\pi_A^{-1}(B_n)$ are not bounded. However, the boundedness of the sets $B_n$ implies the sets $A_r\cap \pi_A^{-1}(B_n)$ are bounded such that $A_r\cap \pi_A^{-1}(B_n)$ has finite volume and $\bd A_r\cap \pi_A^{-1}(B_n)$ has finite surface area for each $r>0$.

As a first step towards a generalization of Theorem~\ref{thm:local-main}, we state the following auxiliary result, which characterizes vague convergence of a family of measures in terms of the vague convergence of their restrictions to (bounded) open sets.

\begin{Lemma} \label{lem:restriction}
Let $\{\mu_r:r\geq 0\}$ be a family of Radon measures on $\R^d$.
\begin{enumerate}
  \item[(i)] Assume that $\mu_r$ converges vaguely to $\mu_0$ as $r \se 0$ and let $G\subset\R^d$ be some open set. Then the restriction $\mu_r|_G$ (given by $\mu_r|_G(C):=\mu_r(C)$, $C\in\sB(G)$) is a Radon measure on $\sB(G)$ for each $r\geq 0$, and
      $
      \mu_r|_G\vlim \mu_0|_G
      $
      as $r\se 0$.
    \item[(ii)] If $\mu_r|_{G}\vlim \mu_0|_{G}$ as $r\se 0$ holds for each open set $G\subset\R^d$, then
      $$
      \mu_r\vlim \mu_0, \text{ as } r\se 0.
      $$
      In fact, for the conclusion to hold, it is enough to require the hypothesis to hold for an increasing sequence of open sets $G$ whose union covers $\R^d$.
\end{enumerate}
\end{Lemma}
\begin{proof}
(i) Let $K\subset G$ be compact. By the criterion \eqref{eq:vag-cond}, the assumed vague convergence $\mu_r\vlim\mu_0$ implies
$$
\limsup_{r\se 0} \mu_r|_G(K)=\limsup_{r\se 0} \mu_r(K)\le \mu_0(K)=\mu_0|_G(K).
$$
Similarly, if $V\subset G$ is open and relatively compact in $G$, then $V$ is relatively compact in $\R^d$
and
$$
\liminf_{r\se 0} \mu_r|_G(V)=\liminf_{r\se 0} \mu_r(V)\ge \mu_0(V)=\mu_0|_G(V).
$$
By \eqref{eq:vag-cond}, these two observations together imply the vague convergence $\mu_r|_G\vlim \mu_0|_G$ as $r\se 0$.

(ii) Fix some increasing sequence of open sets $G_i\subset\R^d$, $i\in\N$ such that $\bigcup_i G_i=\R^d$. By the hypothesis, we have $\mu_r|_{G_i}\vlim\mu_0|_{G_i}$, as $r\se 0$ for each $i\in\N$. To show the vague convergence  $\mu_r\vlim\mu_0$ as $r\se 0$, we use again the criterion \eqref{eq:vag-cond}: Let $K\subset\R^d$ be compact. Then there exists an index $i$ such that $K\subset G_i$ and we infer
$$
\limsup_{r\se 0} \mu_r(K)=\limsup_{r\se 0} \mu_r|_{G_i}(K)\le \mu_0|_{G_i}(K)=\mu_0(K).
$$
Similarly, given some open, relative compact set $V$ in $\R^d$, we can find an index $i$ such that $\overline{V}\subset G_i$ (implying that $V$ is relatively compact in $G_i$) and we get
$$
\liminf_{r\se 0} \mu_r(V)=\liminf_{r\se 0} \mu_r|_{G_i}(V)\ge \mu_0|_{G_i}(V)=\mu_0(V).
$$
Since this is true for any such $K$ and any such $V$, the vague convergence $\mu_r\vlim \mu_0$ follows again from~\cite[Satz 30.2]{Bauer}.
\end{proof}

We will use Lemma~\ref{lem:restriction} to reduce the problem of comparing the full parallel volume and full parallel surface area of an unbounded set (which are infinite measures) to the problem of comparing the restrictions of these measures to suitable open sets on which the measures are finite. The next statement generalizes Theorem~\ref{thm:local-main} to such restrictions.

\begin{Proposition}\label{prop:local-main2}
Let $A\subset\R^d$ be a closed set and $s\in[0,d)$. Let $B\subset\R^d$ be open and bounded and set $\tilde B:=\piAB$. 
Let $\mu$ be a finite Radon measure on $A\cap B$.
Then
\begin{align} \label{eq:mu_weakconv2}
  \mu_r^s|_{\tilde B}(A,\cdot)\wlim\mu \text{ as } r\searrow 0
\end{align}
if and only if
\begin{align} \label{eq:sigma_weakconv2}
  \sigma_r^s|_{\tilde B}(A,\cdot)\wlim\mu \text{ as } r\searrow 0.
 \end{align}
\end{Proposition}
\begin{proof}
The line of proof is very similar to the one of Theorem~\ref{thm:local-main}.
Assume \eqref{eq:mu_weakconv2} holds. By the Portemanteau Theorem,
the weak convergence implies in particular that,
\begin{align}
  \label{eq:mu-continuity2}
  \lim_{r\to 0} \mu_r^s|_{\tilde B}(A,\piAC)=\mu(\piAC)
\end{align}
for all Borel sets $C\subset\R^d$ such that $\mu(\bd \piAC)=0$.

To show the tightness of the family $\{\sigma_r^s|_{\tilde B}(A,\cdot):r\in(0,1]\}$ is more delicate now. Let $(r_n)_{n\in\N}$ be an arbitrary null sequence of radii $r_n>0$. We will first show that the sequence $\{\sigma_{r_n}^s|_{\tilde B}(A,\cdot):n\in\N\}$ is tight.

The weak convergence of the measures $\mu_r^s|_{\tilde B}(A,\cdot)$ implies the tightness of the family $\{\mu_r^s|_{\tilde B}(A,\cdot):r\in(0,1]\}$, that is, for each $\eps>0$ there exists a compact set $K\subset\tilde B$ such that $\mu_r^s|_{\tilde B}(A,\tilde B\setminus K)<\eps$ for all $r\in(0,1]$. Let 
$$
 \tilde K:=\{x\in \R^d: \exists y\in K \text{ such that } x\in [y,\pi_A(y)]\},
$$
that is, let $\tilde K$ be the smallest metrically associated set containing $K$. It is easy to see that $\tilde K$ is a compact subset of $\tilde B$ containing $K$ and that thus
$$
\mu_r^s|_{\tilde B}(A,\tilde B\setminus \tilde K)\leq \mu_r^s|_{\tilde B}(A,\tilde B\setminus K)<\eps
$$
for each $r\in(0,1]$. This implies 
\begin{align}
  \label{eq:tightness-proof}
  \usM^s(A,\tilde B\setminus \tilde K)=\limsup_{r\se 0} \mu_r^s|_{\tilde B}(A,\tilde B\setminus \tilde K)\leq \eps
\end{align}
and since $\tilde K$ is metrically associated with $A$, we can infer from Theorem~\ref{thm:u-rel-contents}, that
$$
\limsup_{n\to\infty} \sigma_{r_n}^s|_{\tilde B}(A,\tilde B\setminus\tilde K)\leq \limsup_{r\se 0} \sigma_r^s|_{\tilde B}(A,\tilde B\setminus\tilde K)=\usS^s(A,\tilde B\setminus \tilde K)
\leq c \eps
$$
with $c:=\frac d{d-s}$.
By the definition of the limsup, there must be some index $n_0$ such that, for each $n\geq n_0$, we have
$$
\sigma_{r_n}^s|_{\tilde B}(A,\tilde B\setminus\tilde K)\le 2c\eps.
$$
(The inner regularity of the measures $\sigma_{r_n}$ allows to enlarge $\tilde K$ such that this inequality holds for all $n$.)
 Since this works for any $\eps>0$, the family $\{\sigma_{r_n}^s|_{\tilde B}(A,\cdot):n\in\N\}$ is tight. Now, by Prohorov's Theorem, every sequence in this family has a weakly converging subsequence. Let $(r_{n(i)})_{i\in\N}$ be a subsequence such that the measures $\sigma_{r_{n(i)}}^s|_{\tilde B}(A,\cdot)$ converge weakly as $i\to\infty$. Let $\nu$ be the limit measure of this sequence (which is necessarily concentrated on $A\cap B$). We will show that $\nu$ coincides with $\mu$. Since the choice of the subsequence $(r_{n(i)})$ was arbitrary, the limit measure must then be the same for every such sequence, which implies the weak convergence $\sigma_{r_n}^s|_{\tilde B}(A,\cdot)\wlim\mu$ as $n\to\infty$ as desired. Furthermore, since the sequence $(r_n)$ was arbitrary, we can even conclude the weak convergence $\sigma_r^s|_{\tilde B}(a,\cdot)\wlim \mu$ as $r\se 0$.

To show the equality $\nu=\mu$, by Corollary~\ref{cor:meas-uniqueness} (iii), it suffices to prove that $\nu(\piAC)=\mu(\piAC)$ for all Borel sets $C\subset B$ such that $\mu(\bd \piAC)=0$. So let $C$ such a set. Since, by the observation \eqref{eq:mu-continuity2} above, the relative Minkowski content
$$\sM^s(A,\piAC)=\lim_{r\se 0} \mu_r^s|_{\tilde B}(A,\piAC)$$
exists and coincides with $\mu(\piAC)$, we can infer from Theorem~\ref{thm:relMmeas} (if $\mu(\piAB)>0$) or otherwise from Corollary~\ref{cor:rel-contents-zero}, that also the limit
$$
\sS^s(A,\piAC)=\lim_{r\se 0} \sigma_r^s|_{\tilde B}(A,\piAC)=\lim_{i\to\infty} \sigma_{r_{n(i)}}^s|_{\tilde B}(A,\piAC)=\nu(\piAC)
$$
exists and equals $\mu(\piAC)$. This shows $\nu=\mu$ and completes the proof of the implication \eqref{eq:mu_weakconv2} $\Rightarrow $ \eqref{eq:sigma_weakconv2}.

The proof of the reverse implication is almost the same with the roles of the families $\mu_r^s|_{\tilde B}(A,\cdot)$ and $\sigma_r^s|_{\tilde B}(A,\cdot)$ interchanged and with the constant $c$ above replaced by $1$ (given again by Theorem~\ref{thm:u-rel-contents}).
\end{proof}

We are now ready to state and prove the main result of this section, the announced equivalence of local Minkowski content $\mu^s(A,\cdot)$ and local S-content $\sigma^s(A,\cdot)$ for arbitrary (possibly unbounded) closed sets $A\subset\R^d$. We emphasize that according to the statement it is enough to assume the existence of one of the two local contents, $\mu^s(A,\cdot)$ or $\sigma^s(A,\cdot)$, to conclude the existence of other one and their coincidence.

\begin{Theorem}\label{thm:local-main3}
Let $A\subset\R^d$ be a closed set and $s\in[0,d)$. 
Let $\mu$ be some Radon measure concentrated on $A$. Then
\begin{align} \label{eq:mu_weakconv3}
  \mu_r^s(A,\cdot)\vlim\mu \text{ as } r\searrow 0
\end{align}
if and only if
\begin{align} \label{eq:sigma_weakconv3}
  \sigma_r^s(A,\cdot)\vlim\mu \text{ as } r\searrow 0.
 \end{align}
If \eqref{eq:mu_weakconv3} holds and $\mu\neq 0$, then $s$ is the Minkowski dimension of $A$ (in the sense of the definition given in \eqref{eq:def-mink-unbd}).
\end{Theorem}
\begin{proof}
  Assume \eqref{eq:mu_weakconv3} holds. Let $(B_i)_{i\in\N}$ be an increasing sequence of bounded open sets such that $\bigcup_i B_i=\R^d$ and $\mu(\bd B_i)=0$ for each $i\in\N$. (This last condition can easily be satisfied. Consider e.g.\ a family $\{B(r): r>0\}$ of concentric balls $B(r)$ of radius $r$. Their boundaries $\bd B(r)$ are pairwise disjoint Borel sets and thus at most countably many of them can have positive $\mu$-measure. The sets $B_i$ can be chosen from those balls with $\mu(\bd B(r))=0$.) Let $G_i:=\pi^{-1}_A(B_i)$. Then $G_i$ is open, $A\cap G_i=A\cap B_i$ is bounded and thus $A_r\cap G_i$ is open and bounded for any $r>0$. Since $\mu$ is concentrated on $A$, the additional assumption $\mu(\bd B_i)=0$ implies $\mu(\bd G_i)=\mu(\bd G_i\cap A)\leq\mu(\bd B_i\cap A)=0$.

 Writing $\mu_r:=\mu_r^s(A,\cdot)$, by Lemma~\ref{lem:restriction}, we have the vague convergence $\mu_r|_{G_i}\vlim \mu|_{G_i}$ as $r\se 0$ for each $i\in\N$. The additional assumption implies that there is in fact no loss of mass, i.e., we have $\lim_{r\se 0}\mu_r|_{G_i}(G_i)=\mu|_{G_i}(G_i)$ for the total masses of the restrictions. To see this observe that on the one hand we have
 $$
 \limsup_{r\se 0}\mu_r(G_i)=\limsup_{r\se 0} \mu_r(A_1\cap G_i)\leq\limsup_{r\se 0} \mu_r(\overline{A_1\cap G_i})\leq \mu(\overline{A_1\cap G_i}),
 $$
 due to the compactness of the set $\overline{A_1\cap G_i}$, and the latter expression satisfies $\mu(\overline{A_1\cap G_i})\leq \mu(\overline{G_i})=\mu(G_i)$, since $\bd G_i$ carries no mass. On the other hand we have
 $$
 \liminf_{r\se 0}\mu_r(G_i)=\liminf_{r\se 0} \mu_r(A_1\cap G_i)\ge \mu(A_1\cap G_i)=\mu(A\cap G_i)=\mu(G_i),
 $$
 since $A_1\cap G_i$ is open and relatively compact. Hence $\lim_{r\se 0}\mu_r(G_i)=\mu(G_i)$, which proves the asserted.

 The preservation of mass of a vaguely convergent sequence of measures is equivalent to its weak convergence, cf.\ e.g.~\cite[Satz 30.8]{Bauer}. Hence, we have in fact the weak convergence $\mu_r|_{G_i}\wlim\mu|_{G_i}$ as $r\se 0$ for each $i\in\N$.

 Now, writing $\sigma_r:=\sigma_r^s(A,\cdot)$ for short, Proposition~\ref{prop:local-main2} implies the weak convergence $\sigma_r|_{G_i}\wlim\mu|_{G_i}$ as $r\se 0$ for each $i\in\N$, which, by Lemma~\ref{lem:restriction} (ii), is sufficient to conclude the vague convergence $\sigma_r\vlim\mu$ as $r\se 0$.

 This proves one direction of the equivalence. The other one follows similarly with the roles of the measures $\mu_r^s(A,\cdot)$ and $\sigma_r^s(A,\cdot)$ interchanged.

 Assume now we are in the situation that \eqref{eq:mu_weakconv3} (and thus \eqref{eq:sigma_weakconv3}) holds for some locally finite Radon measure $\mu\neq 0$ and some $s$. Let $(B_i)_{i\in\N}$ be an increasing sequence of bounded open sets with $\mu(\bd \pi_A^{-1}(B_i))=0$ for each $i\in\N$. The latter property implies that
 $$
 \sM^s(A,\pi_A^{-1}(B_i))=\lim_{r\se 0} \mu^s_r(A,\pi_A^{-1}(B_i))=\mu(\pi_A^{-1}(B_i))<\infty
 $$
 because of the local finiteness of $\mu$. This implies $\udim_M A\leq s$.
 Moreover, since $\mu\neq 0$, there exists some $i_0\in\N$ such that $\sM^s(A,\pi_A^{-1}(B_i))=\lim_{r\se 0} \mu^s_r(A,\pi_A^{-1}(B_i))=\mu(\pi_A^{-1}(B_i))>0$ for any $i\geq i_0$. This implies $\lsM^t(A,\pi_A^{-1}(B_i))=\infty$ for any $t<s$ and thus $\ldim_M A\geq s$. This shows that $\dim_M A$ exists and equals $s$ completing the proof of Theorem~\ref{thm:local-main3}.
 \end{proof}

%

{\bf Properties of local Minkowski and S-content.} In case local Minkowski content and local S-content exist, these measures inherit some of the properties from their defining sequence of measures. It is easy to see that some form of motion invariance and homogeneity of volume and surface area, respectively, survive in the limit: If $\mu^s(A,\cdot)$ exists for some closed set $A\subset\R^d$, then, for any Euclidean motion $g$ and any Borel set $B$,
\begin{align*}
  \mu^s(gA,gB)=\mu^s(A,B).
\end{align*}
Similarly, if $\sigma^s(A,\cdot)$ exists, then
\begin{align*}
  \sigma^s(gA,gB)=\sigma^s(A,B).
\end{align*}
Furthermore, the $s$-dimensional contents are homogeneous of degree $s$ if they exist, that is,
\begin{align*}
  \mu^s(\lambda A,\lambda B)=\lambda^s \mu^s(A,B) \quad \text{ and } \quad \sigma^s(\lambda A,\lambda B)=\lambda^s \sigma^s(A,B),
\end{align*}
 for any $\lambda>0$ and any Borel set $B$.

The most important property inherited from the defining sequence is that local Minkowski content and local S-content are locally determined. This property is in a way the ultimate reason, why all the localization results presented here hold.

\begin{Proposition}\label{thm:local-main2}
Let $A,A'\subset\R^d$ be closed sets and let $B\subset\R^d$ be an open set such that $A\cap B=A'\cap B$. Let $s\in[0,d)$.
\begin{enumerate}
  \item[(i)] If the local Minkowski contents $\mu^s(A,\cdot)$ and $\mu^s(A',\cdot)$ exist, then
  \begin{align}
    \label{eq:locality1} \mu^s(A,C)=\mu^s(A',C)
  \end{align}
  for any Borel set $C\subset B$. 
  That is, we have $\mu^s|_B(A,\cdot)=\mu^s|_B(A',\cdot)$.
  \item[(ii)] Similarly, if the local S-contents $\sigma^s(A,\cdot)$ and $\sigma^s(A',\cdot)$ exist, then
  \begin{align}
    \label{eq:locality1} \sigma^s(A,C)=\sigma^s(A',C)
  \end{align}
  for any Borel set $C\subset B$. 
  That is,
   we have $\sigma^s|_B(A,\cdot)=\sigma^s|_B(A',\cdot)$.
\end{enumerate}
\end{Proposition}
\begin{proof}
Let $(B_i)_{i\in\N}$ be an increasing sequence of open, relatively compact subsets of $B$ such that $\bigcup_i B_i=B$. The relative compactness of ${B_i}$ implies in particular that $B_i\cap A$ is bounded, which ensures that all the sets $A_r\cap \pi_A^{-1}(B_i)$ are bounded. Moreover, we have $\overline{B_i}\subset B$, such that $\delta_i:=d(B_i\cap A,B^c)>0$. Let $G_i:=\pi_A^{-1}(B_i)$. Note that $G_i$ is open.

Let $\mu_r:=\mu^s_r(A,\cdot)$, $\mu:=\mu^s(A,\cdot)$ and similarly $\mu'_r:=\mu^s_r(A',\cdot)$, $\mu':=\mu^s(A',\cdot)$. By Lemma~\ref{lem:restriction}(i), the assumed existence of the local Minkowski contents of $A$ and $A'$ implies the vague convergences $\mu_r|_{G_i}\vlim\mu|_{G_i}$ and $\mu'_r|_{G_i}\vlim\mu'|_{G_i}$ as $r\se 0$. We will show that the limit measures $\mu|_{G_i}$ and $\mu'|_{G_i}$ coincide for each $i\in\N$.

Fix $i\in\N$ for a moment.  We claim that, for any $0<r<\frac{\delta_i}3$,
  \begin{align}
     \label{eq:locality-proof1} A_r\cap G_i=A'_r\cap G_i.
  \end{align}
To see this let $x\in A_r\cap\pi^{-1}_A(B_i)$ and let $y:=\pi_A(x)$. Then $d(x,y)=d(x,A)<r$ and $y\in A\cap B_i$ and therefore, since $B_i\subset B$, $y\in A'\cap B_i$ and $d(x,A')\leq d(x,y)<r$. Moreover, we have $\pi_{A'}(x)=y$. (Assume not. Then there exists $y'\in A'$ such that $d(x,y')<d(x,y)<r$. Since $r<\frac{\delta_i} 3$ and $A\cap B=A'\cap B$, $y'$ belongs to $A$ as well, contradicting that $y=\pi_A(x)$ is the nearest point in $A$.) This means $x\in A'_r\cap\pi^{-1}_{A'}(B_i)$, proving one inclusion of \eqref{eq:locality-proof1}. The reverse inclusion follows by symmetry interchanging the role of $A$ and $A'$.

From \eqref{eq:locality-proof1}, we can immediately infer that
\begin{align}
     \label{eq:locality-proof2} \mu_r|_{G_i}=\mu'_r|_{G_i}
  \end{align}
for any $0<r<\frac{\delta_i} 3$, that is, the two families of measures $\{\mu_r|_{G_i}:0<r<\frac{\delta_i}3\}$ and $\{\mu'_r|_{G_i}:0<r<\frac{\delta_i}3\}$ coincide for small $r$, which implies that the limit measures as $r\se 0$ (which exist by the assumptions) also coincide. Hence we have $\mu|_{G_i}=\mu'|_{G_i}$ for each $i\in\N$.
Since the sets $G_i$ cover $\piAB$, we conclude from part (ii) of Lemma~\ref{lem:restriction}, that
$\mu|_{\piAB}=\mu|_{\piAB}$, which means in fact that $\mu|_B=\mu'|_B$, since $A\cap\piAB =A\cap B$ and the measures $\mu,\mu'$ are concentrated on $A$. This completes the proof of part (i).

The proof of part (ii) follows exactly the same lines. Just observe that the set equivalence \eqref{eq:locality-proof1} implies also that the restricted surface measures $\sigma_r^s|_{G_i}(A,\cdot)$ and $\sigma_r^s|_{G_i}(A',\cdot)$ coincide for $0<r<\frac{\delta_i}3$.
\end{proof}

\section{Average local Minkowski content and S-content}

To complete the picture, we consider also localized versions of averaged contents. For $A,\Omega\subset\R^d$ such that $A\cap \Omega$ is bounded and $s\geq 0$, let
\begin{align} \label{eq:av-rel-MC}
\asM^s(A,\Omega):=\lim_{t\searrow 0} \frac{1}{|\log t|}\int_\delta^1 \frac{\lambda_d(A_r\cap \Omega)}{\kappa_{d-s} r^{d-s}} \frac{dr}{r}
\end{align}
and
\begin{align} \label{eq:av-rel-SC}
\asS^s(A,\Omega):=\lim_{t\searrow 0} \frac{1}{|\log t|}\int_\delta^1 \frac{\Ha^{d-1}(\bd A_r\cap \Omega)}{(d-s)\kappa_{d-s} r^{d-s-1}} \frac{dr}{r},
\end{align}
be the \emph{average relative Minkowski} and \emph{S-content}, respectively, of the set $A$ relative to the set $\Omega$,
whenever these Cesaro averages exist. We write $\alsM^s(A,\Omega)$, $\ausM^s(A,\Omega)$ and $\alsS^s(A,\Omega)$, $\ausS^s(A,\Omega)$ for the corresponding lower and upper average limits.
It is rather easy to see that the following general relations hold
\begin{align*}
  \lsM^s(A,\Omega)\leq \alsM^s(A,\Omega)&\leq \ausM^s(A,\Omega)\leq\usM^s(A,\Omega),\\
  \lsS^s(A,\Omega)\leq \alsS^s(A,\Omega)&\leq \ausS^s(A,\Omega)\leq\usS^s(A,\Omega).
\end{align*}

Our first point is the following localization of \cite[Lemma 4.6]{RW09}.

\begin{Lemma} \label{lem:av-contents}
  Let $A\subset\R^d$ be closed and let $B\subset\R^d$ be some Borel set such that $A\cap B$ is bounded. Let $0\leq s<d$.
  \begin{enumerate}
  \item[(i)] Then
  $$
      \ausM^s(A,\piAB)\geq\ausS^s(A,\piAB) \quad \text{ and }\quad \alsM^s(A,\piAB)\geq\alsS^s(A,\piAB).
   $$

   \item[(ii)]   If $\usM^s(A,\piAB)<\infty$, then
  $$
      \ausM^s(A,\piAB)=\ausS^s(A,\piAB) \text{ and } \alsM^s(A,\piAB)=\alsS^s(A,\piAB).
      $$
      \end{enumerate}
    In particular, if $A\subset\R^d$ is compact and $\usM^s(A)<\infty$, then the equalities in (ii) hold for any Borel set $B$.

 \end{Lemma}
 \begin{proof}
  We adapt the argument in the proof of \cite[Lemma 4.6]{RW09}:
  For $r>0$, write $S(r):=\Ha^{d-1}(\bd A_r\cap\piAB)$. (Note that, by the boundedness of $A\cap B$, the set $\bd A_r\cap\piAB$ is a bounded $(n-1)$-rectifiable set with finite measure. Hence $S(r)$ is a finite number for each fixed $r>0$.) For $0<t\le 1$, define
\[
v(t):=\int_t^1 \frac{V_{A,B}(r)}{\kappa_{d-s}r^{d-s}} \frac{dr}{r} \quad\text{ and }\quad w(t):=\int_t^1 \frac{S(r)}{(d-s)\kappa_{d-s}r^{d-s-1}} \frac{dr}{r}.
\]
Our first claim is that
\begin{equation} \label{eq:v-w}
v(t)= w(t) +\frac{1}{d-s} \frac{V_{A,B}(t)}{\kappa_{d-s} t^{d-s}} - \frac{1}{(d-s)\kappa_{d-s}} V_{A,B}(1).
\end{equation}
By  Theorem~\ref{thm:loc-diff-points}, we have
\[
v(t)=\int_t^1 \int_0^r S(\rho) d\rho \frac{dr}{\kappa_{d-s}r^{d-s+1}},
\]
and, interchanging the order of integration, this leads to
\begin{eqnarray*}
v(t)&=&\frac 1{\kappa_{d-s}}\left[
\int_0^t S(\rho)  \int_t^1  \frac{dr}{r^{d-s+1}} d\rho + \int_t^1 S(\rho) \int_\rho^1 \frac{dr}{r^{d-s+1}} d\rho
\right]\\
&=& \frac 1{(d-s)\kappa_{d-s}}\left[ V_{A,B}(t) \left(\frac 1{t^{d-s}}-1\right) + \int_t^1 S(\rho) \left(\frac 1{\rho^{d-s}}-1\right) d\rho\right]\\
&=& \frac 1{(d-s)\kappa_{d-s}} \left(\frac{V_{A,B}(t)}{ t^{d-s}}-V_{A,B}(t)-V_{A,B}(1)+V_{A,B}(t)\right) +  w(t).
\end{eqnarray*}
Here we have used again the relation $V_{A,B}(r)=\int_0^r S(\rho) d\rho$ as well as the assumption $s<d$. This proves \eqref{eq:v-w}.

Observe that the third term on the right in \eqref{eq:v-w} is constant. It vanishes, when dividing by $|\log t|$ and taking the limit as $t\to\infty$. The second term is non-negative.
Let $(t_n)$ be a null sequence, such that
\[
\lim_{n\to\infty} \frac{w(t_n)}{|\log t_n|}=\ausS^s(A,\piAB).
\]
Then
\[
\ausM^s(A,\piAB)\ge \limsup_{n\to\infty} \frac{v(t_n)}{|\log t_n|}\ge \ausS^s(A,\piAB).
\]
Similarly the inequality $\alsM^s(A,\piAB)\ge \alsS^s(A,\piAB)$ is obtained by choosing a sequence $(\tilde t_n)$ such that $\alsM^s(A,\piAB)$ is attained, completing the proof of (i).

If $\usM^s(A,\piAB)<\infty$ holds, then the second term on the right in \eqref{eq:v-w} is bounded by a constant. Hence, it vanishes when dividing by $|\log t|$ and taking the limit as $t\to\infty$. The equalities stated in (ii) follow at once.

  For the last assertion, in which $A$ is assumed to be bounded note that, by the monotonicity of relative Minkowski contents, we have, for any Borel set $B\subset\R^d$, $\usM^s(A,\piAB)\leq \usM^s(A)<\infty$. Hence (ii) applies.
\end{proof}

\begin{Theorem} \label{thm:av-rel-contents}
  (Average relative Minkowski contents and S-contents)
  \begin{itemize}
     \item[(i)]
     Let $A\subset\R^d$ be a closed set and $B\subset\R^d$ some Borel set such that $A\cap B$ is bounded and $\usM^D(A,\piAB)<\infty$ for some $D\in[0,d)$. Let $M\in[0,\infty)$. Then
  $$
  \asM^D(A,\piAB)=M \quad \text{ if and only if } \quad \asS^D(A,\piAB)=M.
  $$
  \item[(ii)] In particular, if $A\subset\R^d$ is compact and $\usM^D(A)<\infty$ for some $D\in[0,d)$, then
  $$
  \asM^D(A,\piAB)=\asS^D(A,\piAB),
  $$
  for any Borel set $B\subset\R^d$ such that one (and thus both) of these two average limits exist.
  \end{itemize}
\end{Theorem}
\begin{proof}
    (i) If $ \asM^D(A,\piAB)=M$ holds, then it follows from Lemma~\ref{lem:av-contents} (ii), that $\asS^D(A,\piAB)$ exists and equals $M$, and vice versa.

    (ii) This follows from (i), taking again into account that, by the monotonicity of the relative contents, $\usM^s(A,\piAB)\leq \usM^s(A)<\infty$ for any Borel set $B\subset\R^d$.
\end{proof}

In analogy with the results for local Minkowski content and local S-content in Theorem~\ref{thm:local-main}, we can now establish a corresponding relation for the averaged local contents, which are introduced as follows. First, define the rescaled average local parallel volume of a closed set $A\subset\R^d$ for any $t>0$ by
\begin{align} \label{eq:av_loc_vol}
\widetilde{\mu}_t^s(A,\cdot)&:=\frac 1{|\log t|}\int_t^1\frac{\lambda_d(A_r\cap\cdot)}{\kappa_{d-s} r^{d-s}}\frac{dr}r,
\end{align}
where $s\in[0,d]$. It is obvious that, for any $s\geq 0$ and $t>0$, $\mu_t^s(A,\cdot)$ is a locally finite measure concentrated on the closed parallel set $A_{\leq t}$ of $A$. Moreover, it is easy to check that these measures are inner regular. Hence the family $\{\widetilde{\mu}_r^s(A,\cdot):r\in (0,1])\}$ consists of Radon measures. The ($s$-dimensional) \emph{average local Minkowski content} $\widetilde{\mu}^s(A,\cdot)$ of $A$ is then defined as the vague limit of these measures as $r\searrow 0$ (provided it exists). 
Note that (in case it exists) $\widetilde{\mu}^s(A,\cdot)$ is a measure concentrated on $A$. Note also that the measures $\widetilde{\mu}^s_t(A,\cdot)$ depend in fact also on the upper bound of the integration interval in the definition, which we have set to 1. Using any other upper bound $u>0$ instead of $1$, will produce a different family of measures. However, the existence of the limit and even the limit measure $\widetilde\mu^s(A,\cdot)$ are not affected by the choice of $u$.
It will be clear from the statement below that for a nontrivial limit $\widetilde\mu^s(A,\cdot)$ to exist, one has to choose $s=\dim_M A$ (in the sense of the definition in \eqref{eq:def-mink-unbd} in case of an unbounded $A$).

Similarly, for any closed set $A\subset\R^d$ and $s\geq 0$, the vague limit $\widetilde{\sigma}^s(A,\cdot)$ of the appropriately rescaled surface measures
\begin{align*}
\widetilde\sigma_t^s(A,\cdot):=\frac1{|\log t|}\int_t^1\frac{\Ha^{d-1}(\bd A_r\cap \cdot)}{\kappa_{d-s} (d-s) r^{d-s-1}}\frac{dr}r, t>0
\end{align*}
as $t\searrow 0$, provided it exists, will be regarded as the ($s$-dimensional) \emph{average local S-content} of $A$, denoted by $\widetilde\sigma^s(A,\cdot)$.
Note that, similarly as for the local parallel volume, the local surface area $\widetilde{\sigma}^s_t(A,\cdot)$  is a Radon measure on $A_{\leq t}$ for any $s\geq 0$ and $t>0$.

Recall that for a bounded set $A\subset\R^d$ and $D:=\udim_M A$ the finiteness of the upper Minkowski content $\usM^D(A)$ is not ensured in general, and has to be assumed for certain results. The corresponding condition for unbounded sets $A$ (still with $D:=\udim_M A$) reads as follows: assume there exists an increasing sequence $(B_n)_{n\in\N}$ of bounded sets such that
\begin{align}
  \label{eq:upper-Mink-finite}\usM^D(A,\pi_A^{-1}(B_n))<\infty \text{ for each } n\in\N.
\end{align}
This condition is required in the next statement to ensure that the average local contents of unbounded sets are well enough behaved. In particular, we need it to be able to apply Lemma~\ref{lem:av-contents} and Theorem~\ref{thm:av-rel-contents}. For bounded sets this condition obviously reduces to the finiteness of the upper Minkowski content $\usM^D(A)<\infty$.

%
%

\begin{Theorem}\label{thm:local-average}
Let $A\subset\R^d$ be a closed set with $D:=\udim_M A\in[0,d)$. Assume that condition \eqref{eq:upper-Mink-finite} holds. Let $\mu$ be a Radon measure on $A$. Then
\begin{align} \label{eq:mu_avweakconv}
  \widetilde\mu_t^D(A,\cdot)\vlim\mu \text{ as } t\searrow 0
\end{align}
if and only if
\begin{align} \label{eq:sigma_avweakconv}
  \widetilde\sigma_t^D(A,\cdot)\vlim\mu \text{ as } t\searrow 0.
 \end{align}
If \eqref{eq:mu_avweakconv} (and thus \eqref{eq:sigma_avweakconv}) holds and $\mu\neq 0$, then $D=\dim_M A$.
Moreover, if \eqref{eq:mu_avweakconv} holds and $A$ is compact (which implies that $\mu$ is finite), then the total mass $\mu(A)$ necessarily coincides with the average Minkowski content and the average S-content of $A$, i.e.,
$$
\mu(A)=\asM^D(A)=\asS^D(A).
$$
\end{Theorem}

 For the proof of Theorem~\ref{thm:local-average}, we need an analogue of Proposition~\ref{prop:local-main2} for restrictions of average local contents to an open set.

\begin{Proposition}\label{thm:local-average2}
Let $A\subset\R^d$ be a closed set and $s\in[0,d)$. Let $B\subset\R^d$ be open and bounded and set $\tilde B:=\piAB$. Assume that $\usM^s(A,\tilde B)<\infty$. 
Let $\mu$ be a finite Radon measure on $A\cap B$. 
Then
\begin{align} \label{eq:mu_avweakconv2}
  \widetilde\mu_r^s|_{\tilde B}(A,\cdot)\wlim\mu \text{ as } r\searrow 0
\end{align}
if and only if
\begin{align} \label{eq:sigma_avweakconv2}
  \widetilde\sigma_r^s|_{\tilde B}(A,\cdot)\wlim\mu \text{ as } r\searrow 0.
 \end{align}
\end{Proposition}

\begin{proof}
   The proof follows essentially the argument in the proof of Proposition~\ref{prop:local-main2}. More precisely, the proof is literally the same up to \eqref{eq:tightness-proof} with $\mu_r^s(A,\cdot)$ and $\sigma_r^s(A,\cdot)$ replaced by $\widetilde\mu_r^s(A,\cdot)$ and $\widetilde\sigma_r^s(A,\cdot)$, respectively. Then one can use the assumed finiteness of $\usM^s(A,\tilde B)$ to infer from Lemma~\ref{lem:av-contents}(ii) that
   $$
\limsup_{n\to\infty} \widetilde\sigma_{r_n}^s|_{\tilde B}(A,\tilde B\setminus\tilde K)\leq \limsup_{r\se 0} \widetilde\sigma_r^s|_{\tilde B}(A,\tilde B\setminus\tilde K)=\ausS^s(A,\tilde B\setminus \tilde K)
\leq \eps,
$$
which implies similarly as in the proof of Proposition~\ref{prop:local-main2} that
$
\widetilde\sigma_{r_n}^s|_{\tilde B}(A,\tilde B\setminus\tilde K)\le 2\eps
$
for all $n\in\N$ and therefore the tightness of the family $\{\widetilde \sigma_{r_n}^s|_{\tilde B}(A,\cdot):n\in\N\}$.
Proceeding now as in the proof of Proposition~\ref{prop:local-main2}, one considers the limit measure $\nu$ of a converging subsequence in this family and has to show $\nu=\mu$, for which, by Corollary~\ref{cor:meas-uniqueness}(iii), it suffices to prove that $\nu(\piAC)=\mu(\piAC)$ for all Borel sets $C\subset B$ with $\mu(\bd \piAC)=0$. Let $C$ be such a set. Since $\widetilde \mu_r^s|_{\tilde B}\wlim \mu$ as $r\se 0$, by the Portemanteau Theorem, we have in particular that
$$
\asM^s(A,\piAC)=\lim_{r\se 0} \widetilde\mu_r^s|_{\tilde B}(A,\piAC)=\mu(\piAC),
$$
which, by Theorem~\ref{thm:av-rel-contents}, implies that also the average limit
$$
\asS^s(A,\piAC)=\lim_{r\se 0} \widetilde\sigma_r^s|_{\tilde B}(A,\piAC)=\lim_{i\to\infty} \widetilde\sigma_{r_{n(i)}}^s|_{\tilde B}(A,\piAC)=\nu(\piAC)
$$
exists and equals $\mu(\piAC)$. This shows $\nu=\mu$ and completes the proof of the implication \eqref{eq:mu_avweakconv2} $\Rightarrow $ \eqref{eq:sigma_avweakconv2}.

The reverse implication is proved in the same way with the roles of the families $\widetilde \mu_r^s|_{\tilde B}(A,\cdot)$ and $\widetilde \sigma_r^s|_{\tilde B}(A,\cdot)$ interchanged.
\end{proof}

\begin{proof}[Proof of Theorem~\ref{thm:local-average}]
   The proof of the equivalence of \eqref{eq:mu_avweakconv} and \eqref{eq:sigma_avweakconv} is literally the same as the one in the proof of Theorem~\ref{thm:local-main3} when the measures $\mu_r$ and $\sigma_r$ are defined by $\mu_r:=\widetilde\mu_r^s(A,\cdot)$ and $\sigma_r:=\widetilde{\sigma}_r^s(A,\cdot)$ for $r>0$ instead and when Proposition~\ref{thm:local-average2} is applied instead of Proposition~\ref{thm:local-main2} in the relevant place.

   To see that $s=\dim_M A$ it is enough to show that $\ldim_M A\geq s$ (since $\udim_M A=s$ was assumed). For this let $(B_i)_{i\in\N}$ be an increasing sequence of bounded open sets. The $B_i$ are relatively compact and so,
 since $\mu\neq 0$, the vague convergence implies there exists some $i_0\in\N$ such that $\alsM^s(A,\pi_A^{-1}(B_i))=\liminf_{t\se 0} \widetilde\mu^s_t(A,\pi_A^{-1}(B_i))\geq\mu(\pi_A^{-1}(B_i))>0$ for any $i\geq i_0$. This implies $\lsM^t(A,\pi_A^{-1}(B_i))=\infty$ for any $t<s$ and therefore $\ldim_M A\geq s$. Thus $\dim_M A$ exists and equals $s$.

 The last assertion follows from the fact that in case of a compact set $A$, the vague convergence is in fact weak convergence such that the mass is preserved.
\end{proof}

\section{Applications}

We briefly discuss some applications of the results of the previous sections.

\paragraph{\bf Self-similar sets.} Let $K\subset\R^d$ be a self-similar set satisfying the open set condition (OSC). That is, there is an iterated function system $\{S_1,S_2,\ldots,S_N\}$ consisting of contracting similarities $S_i:\R^d\to\R^d$ with contraction ratios $0<r_i<1$, $i=1,\ldots,N$ such that $K$ is the unique nonempty compact set satisfying the invariance relation $K=\bigcup_{i=1}^N S_iK$.
Let $D:=\dim_M K$ be the Minkowski dimension of $K$ (which is well known to coincide with its Hausdorff dimension) and let
$$
\mu_K:=\Ha^D(K)^{-1}\cdot \Ha^D|_K
$$
be the normalized $D$-dimensional Hausdorff measure on $K$.

In \cite{Winter08}, it was shown that for such self-similar sets $K$ of nonlattice type the local Minkowski content $\mu^D(K,\cdot)$ exists, that is,
the weak limit of the measures $\mu_r^D(K,\cdot)$ as $r\searrow 0$ exists and coincides with the measure $\sM^D(K)\cdot \mu_K$. Note that such $K$ are well known to be Minkowski measurable, i.e. $0<\sM^D(K)<\infty$.

By Theorem~\ref{thm:local-main}, we can now infer immediately, that also the local S-content, $\sigma^D(K,\cdot)$ of the set $K$ exists and is given by
$$
\sigma^D(K,\cdot)=\sM^D(K)\cdot \mu_K=\sS^D(K)\cdot \mu_K.
$$
This result was obtained in \cite{RW09} by a direct computation of the weak limit, but it is now an immediate consequence of the general Theorem~\ref{thm:local-main}. Similarly, the existence of the average local Minkowski content $\widetilde{\mu}^D(A,\cdot)$ was shown in \cite{Winter08} for any self-similar set $A\subset\R^d$ satisfying OSC and the existence of the average local S-content $\widetilde{\sigma}^D(A,\cdot)$ for such sets was shown in \cite{RW09}. The latter can now be deduced directly from the former using Theorem~\ref{thm:local-average}. 

\paragraph{\bf Self-conformal sets: $C^{1+\alpha}$-images of self-similar sets.} Let $K\subset\R^d$ be a self-similar set as above but assume additionally that the strong separation condition (SSC) is satisfied, i.e.\ $S_iF\cap S_j F=\emptyset$ for all $i\neq j$. Let $g:U\to\R^d$ be a conformal diffeomorphism defined on an open set $U$ containing the $\frac 12$-parallel set $K_{1/2}$ of $K$ and assume that $|g'|$ is $\alpha$-H\"older continuous for some $\alpha>0$. Then the set $F:=g(K)$ satisfies the invariance relation
$$
F=\bigcup_{i=1}^N g S_i g^{-1} F,
$$
in which the maps $\Psi_i:=g S_i g^{-1}$ are conformal but not necessarily contractions. However, it can be shown that some iterate of the system $\{\Psi_1,\ldots,\Psi_N\}$ consists of contractions and $F$ is thus the invariant set of a conformal IFS satisfying the SSC, see e.g.\cite{KK12,MU96} for details.

In \cite[Corollary~1.15]{FK12}, it was shown that for such $C^{1+\alpha}$-images $F$ of a self-similar set $K$, the average local Minkowski content $\widetilde{\mu}^D(F,\cdot)$ exists and in the nonlattice case also the local Minkowski content $\mu^D(F,\cdot)$. More precisely, the authors proved that the average Minkowski content $\asM^D(F)$ of $F$ exists (and is positive and finite) and given explicitly in terms of $K$ and the function $g$ by
\begin{align}
   \label{eq:Mink-F-B}
   \asM^D(F)=\asM^D(K)\cdot \int_K |g'|^D d\mu_K.
\end{align}
Moreover, this relation localizes in the sense that $\widetilde{\mu}^D(F,\cdot)$ is absolutely continuous with respect to the push forward $g_*\widetilde{\mu}^D(K,\cdot)$ of $\widetilde{\mu}^D(K,\cdot)$ (given by $g_*\widetilde{\mu}^D(K,B):=\widetilde{\mu}^D(K,g^{-1}(B))=\asM^D(K)\cdot \mu_K(g^{-1}(B))$, $B\subset\sB(\R^d)$, cf.\ the self-similar case above) with Radon-Nikodym derivative given by
$$
\frac{d\widetilde{\mu}^D(F,\cdot)}{dg_*\widetilde{\mu}^D(K,\cdot)}=|g'\circ g^{-1}|^D.
$$
In other words,
$$
\widetilde{\mu}^D(F,B)=\int_B |g'\circ g^{-1}|^D dg_*\widetilde{\mu}^D(K,\cdot)=\asM^D(K)\cdot\int_{g^{-1}(B)} |g'|^D d\mu_K,
$$
for any Borel set $B\subset\R^d$, or, using the above relation \eqref{eq:Mink-F-B},
$$
\widetilde{\mu}^D(F,B)=\asM^D(F)\frac{\int_{g^{-1}(B)} |g'|^D d\mu_K}{\int_K |g'|^D d\mu_K}.
$$
That is, the average local Minkowski content of $F$ exists and is a multiple of the $D$-conformal measure on $F$.
For nonlattice sets $F$, also the local Minkowski content $\mu^D(F,\cdot)$ exists and all of the above relations hold for $\mu^D(F,\cdot)$ instead of $\widetilde{\mu}^D(F,\cdot)$ (and $\sM^D(F)$ instead of $\asM^D(F)$).
In \cite{FK12}, nothing is said about the S-contents or the asymptotic behaviour of the parallel surface area. Combining the results of \cite{FK12} with Theorems~\ref{thm:local-main} and \ref{thm:local-average}, we can now immediately derive analogous results about the local S-contents of such self-conformal sets.

\begin{Theorem}
  Let a self-conformal set $F\subset\R^d$ be the $C^{1+\alpha}$-image $F=g(K)$ of some self-similar set $K\subset\R^d$ as described above satisfying the SSC and let $D=\dim_M F=\dim_M K$. 
  \begin{enumerate}
    \item[(i)]  Then the average S-content $\asS^D(F)$ exists and coincides with $\asM^D(F)$. Moreover, the
  average local S-content $\widetilde{\sigma}^D(F,\cdot)$ exists and coincides with a multiple of the $D$-conformal measure $\mu_F$ on $F$. More precisely,
  $$
  \widetilde{\sigma}^D(F,\cdot)=\asS^D(F)\cdot \mu_F=\asM^D(F)\cdot\mu_F=\widetilde{\mu}^D(F,\cdot).
  $$
  \item[(ii)] If the underlying IFS generating $K$ is nonlattice, then the S-content $\sS^D(F)$ and also the local S-content exist and coincide with their average counterparts.
  \end{enumerate}
\end{Theorem}

\paragraph{\bf Self-conformal sets on the real line.} In \cite{KK12}, Ke\ss eb\"ohmer and Kombrink have studied general self-conformal sets on the real line satisfying the open set condition. They have shown the existence of both the average local Minkowski content and the average local S-content and, in the nonlattice case, also the existence of the local Minkowski content and the local S-content, cf. \cite[Theorem~2.11]{KK12}. With Theorems~\ref{thm:local-main} and \ref{thm:local-average}, respectively, we can now conclude the existence of one of these local contents directly from the existence of the other one and similarly for their average versions. However, as observed in \cite{KK12}, in the one-dimensional case, a relation between local S-content and local Minkowski content can be established more easily directly from the global relation \eqref{eq:equal-of-contents} and the fact that for $b\in\R\setminus F$ the equality $F_\eps\cap(-\infty,b]=(F\cap(-\infty,b])_\eps$ holds for $\eps>0$ sufficiently small.

Completely analogous remarks apply to the (average) local contents of the limit sets of the graph directed conformal iterated function systems studied in \cite{KK15}.

\begin{Remark}
  Note that the existence of the local Minkowski content implies the existence of the Minkowski content of a set but not vice versa, and similarly for the S-contents. Examples of sets (in $\R$) for which the Minkowski content exists but not the local Minkowski content have been discussed in \cite[Corollary~2.18]{KK12}, where some $C^{1+\alpha}$ images $F=g(K)$ of a lattice self-similar set $K$ are considered with $g$ chosen in such a way that $F$ is Minkowski measurable (while $K$ is not).
\end{Remark}

\paragraph{\bf Self-conformal sets in $\R^d$} In \cite{bohl13}, Bohl studies general self-conformal sets in $\R^d$ and establishes the existence of the average local Minkowski content as well as the existence of the average local S-content (along with the existence of averaged fractal curvature measures under additional assumptions), see also \cite{Kom11} for some earlier results in some special cases. Again, Theorem~\ref{thm:local-average} allows to derive the existence of the average local S-content directly
from the existence of the average local Minkowski content.

The existence of (non-averaged) local Minkowski- and S-content in the nonlattice case (and its nonexistence in the lattice case) seems to be open in general. Under additional geometric assumptions, Kombrink \cite{Kom11} has shown the existence of the local Minkowski content for nonlattice self-conformal sets in $\R^d$, cf.~\cite[Theorem~2.29]{Kom11}. 
Applying Theorem~\ref{thm:local-main}, we can immediately derive the existence of the local S-content under the same assumptions.

\begin{Theorem}
  Let $F\subset\R^d$ be the invariant set of a nonlattice self-conformal IFS $\{\psi_1,\ldots,\psi_N\}$ and let $D:=\dim_M F$. Assume that $F$ satisfies the open set condition for an open set $O$ such that $\bd O\subset F$ and $\udim_M(\bd O)<D$.  Assume further that there is some constant $\gamma>0$ such that the expression $\eps^{D-d-\gamma}\lambda_d(F_\eps\cap G)$ is bounded as $\eps\searrow 0$, where $G:=O\setminus \cup_i \overline{\psi_i(O)}$.
  Then the S-content $\sS^D(F)$ exists and coincides with $\sM^D(F)$. Moreover, the
  local S-content $\sigma^D(F,\cdot)$ exists and coincides with a multiple of the $D$-conformal measure $\mu_F$ on $F$. More precisely,
  $$
  \sigma^D(F,\cdot)=\sS^D(F)\cdot \mu_F=\sM^D(F)\cdot\mu_F.
  $$
\end{Theorem}

\bigskip

\paragraph{\bf Acknowledgements} 
This research was partially supported by the German Science Foundation (DFG), project no.\ WI 3264/2-2 as well as by the DFG research unit ``Geometry of spatial random systems''. I thank Jan Rataj and Martina Z\"ahle for fruitful discussions and some valuable comments on an earlier version of the manuscript.

\bibliographystyle{abbrv}
\bibliography{loc-par}

\begin{thebibliography}{10}

\bibitem{Backes1}
A.~R. Backes.
\newblock A new approach to estimate lacunarity of texture images.
\newblock {\em Pattern Recognition Letters}, 34(13):1455--1461, 2013.

\bibitem{Bauer}
H.~Bauer.
\newblock {\em Ma\ss - und {I}ntegrationstheorie}.
\newblock de Gruyter Lehrbuch. [de Gruyter Textbook]. Walter de Gruyter \& Co.,
  Berlin, second edition, 1992.

\bibitem{bohl13}
T.~J. Bohl.
\newblock Fractal curvatures and {M}inkowski content of self-conformal sets.
\newblock {\em Preprint 2012}.
\newblock \arxiv{1211.3421}.

\bibitem{BZ13}
T.~J. Bohl and M.~Z{\"a}hle.
\newblock Curvature-direction measures of self-similar sets.
\newblock {\em Geom. Dedicata}, 167:215--231, 2013.

\bibitem{els}
J.~Elstrodt.
\newblock {\em Ma\ss- und {I}ntegrationstheorie}.
\newblock Springer-Lehrbuch. [Springer Textbook]. Springer-Verlag, Berlin,
  fourth edition, 2005.
\newblock Grundwissen Mathematik. [Basic Knowledge in Mathematics].

\bibitem{F69}
H.~Federer.
\newblock {\em Geometric measure theory}.
\newblock Die Grundlehren der mathematischen Wissenschaften, Band 153.
  Springer-Verlag New York Inc., New York, 1969.

\bibitem{FK12}
U.~Freiberg and S.~Kombrink.
\newblock Minkowski content and local {M}inkowski content for a class of
  self-conformal sets.
\newblock {\em Geom. Dedicata}, 159:307--325, 2012.

\bibitem{Gatzouras2000}
D.~Gatzouras.
\newblock Lacunarity of self-similar and stochastically self-similar sets.
\newblock {\em Trans. Amer. Math. Soc.}, 352(5):1953--1983, 2000.

\bibitem{HR12}
O.~Honzl and J.~Rataj.
\newblock Almost sure asymptotic behaviour of the {$r$}-neighbourhood surface
  area of {B}rownian paths.
\newblock {\em Czechoslovak Math. J.}, 62(137)(1):67--75, 2012.

\bibitem{HLW04}
D.~Hug, G.~Last, and W.~Weil.
\newblock A local {S}teiner-type formula for general closed sets and
  applications.
\newblock {\em Math. Z.}, 246(1-2):237--272, 2004.

\bibitem{KLV}
A.~K{\"a}enm{\"a}ki, J.~Lehrb{\"a}ck, and M.~Vuorinen.
\newblock Dimensions, {W}hitney covers, and tubular neighborhoods.
\newblock {\em Indiana Univ. Math. J.}, 62(6):1861--1889, 2013.

\bibitem{KK12}
M.~Kesseb{\"o}hmer and S.~Kombrink.
\newblock Fractal curvature measures and {M}inkowski content for self-conformal
  subsets of the real line.
\newblock {\em Adv. Math.}, 230(4-6):2474--2512, 2012.

\bibitem{KK15}
M.~Kesseb{\"o}hmer and S.~Kombrink.
\newblock Minkowski content and fractal {E}uler characteristic for conformal
  graph directed systems.
\newblock {\em J. Fractal Geom.}, 2(2):171--227, 2015.

\bibitem{Kneser}
M.~Kneser.
\newblock \"{U}ber den {R}and von {P}arallelk\"orpern.
\newblock {\em Math. Nachr.}, 5:241--251, 1951.

\bibitem{Kom11}
S.~Kombrink.
\newblock Fractal curvature measures and minkowski content for limit sets of
  conformal function systems.
\newblock {\em PhD thesis, Universit{\"a}t Bremen}, 2011.
\newblock {http://elib.suub.uni-bremen.de/edocs/00102477-1.pdf}.

\bibitem{LRZbook}
M.~L. Lapidus, G.~Radunovi\'c, and D.~{\v{Z}}ubrini\'c.
\newblock {\em Fractal Zeta Functions and Fractal Drums. Higher-Dimensional
  Theory of Complex Dimensions.}
\newblock Springer International Publishing, New York, 2016.

\bibitem{FGCD}
M.~L. Lapidus and M.~van Frankenhuijsen.
\newblock {\em Fractal {G}eometry, {C}omplex {D}imensions and {Z}eta
  {F}unctions: {G}eometry and spectra of fractal strings}.
\newblock Springer Monographs in Mathematics. Springer, New York, 2006.

\bibitem{Mandelbrot94}
B.~B. Mandelbrot.
\newblock Measures of fractal lacunarity: {M}inkowski content and alternatives.
\newblock In {\em Fractal geometry and stochastics ({F}insterbergen, 1994)},
  volume~37 of {\em Progr. Probab.}, pages 15--42. Birkh\"auser, Basel, 1995.

\bibitem{MU96}
R.~D. Mauldin and M.~Urba{\'n}ski.
\newblock Dimensions and measures in infinite iterated function systems.
\newblock {\em Proc. London Math. Soc. (3)}, 73(1):105--154, 1996.

\bibitem{RW09}
J.~Rataj and S.~Winter.
\newblock On volume and surface area of parallel sets.
\newblock {\em Indiana Univ. Math. J.}, 59(5):1661--1685, 2010.
\newblock \arxiv{0905.3279}.

\bibitem{RW13}
J.~Rataj and S.~Winter.
\newblock Characterization of {M}inkowski measurability in terms of surface
  area.
\newblock {\em J. Math. Anal. Appl.}, 400:120--132, 2013.
\newblock \arxiv{1111.1825}.

\bibitem{RZ12}
J.~Rataj and M.~Z{\"a}hle.
\newblock Curvature densities of self-similar sets.
\newblock {\em Indiana Univ. Math. J.}, 61(4):1425--1449, 2012.

\bibitem{SSW15}
E.~Spodarev, P.~Straka, and S.~Winter.
\newblock Estimation of fractal dimension and fractal curvatures from digital
  images.
\newblock {\em Chaos Solitons Fractals}, 75:134--152, 2015.
\newblock \arxiv{1408.6333}.

\bibitem{Stacho}
L.~L. Stach{\'o}.
\newblock On the volume function of parallel sets.
\newblock {\em Acta Sci. Math. (Szeged)}, 38(3--4):365--374, 1976.

\bibitem{Winter08}
S.~Winter.
\newblock Curvature measures and fractals.
\newblock {\em Dissertationes Math. (Rozprawy Mat.)}, 453:1--66, 2008.

\bibitem{W11-1}
S.~Winter.
\newblock Lower {S}-dimension of fractal sets.
\newblock {\em J. Math. Anal. Appl.}, 375(2):467--477, 2011.
\newblock \arxiv{1003.3776}.

\bibitem{WZ13}
S.~Winter and M.~Z{\"a}hle.
\newblock Fractal curvature measures of self-similar sets.
\newblock {\em Adv. Geom.}, 13(2):229--244, 2013.
\newblock \arxiv{1007.0696}.

\bibitem{Z11}
M.~Z{\"a}hle.
\newblock Lipschitz-{K}illing curvatures of self-similar random fractals.
\newblock {\em Trans. Amer. Math. Soc.}, 363(5):2663--2684, 2011.

\bibitem{Zu05}
D.~{\v{Z}}ubrini{\'c}.
\newblock Analysis of {M}inkowski contents of fractal sets and applications.
\newblock {\em Real Anal. Exchange}, 31(2):315--354, 2005/06.

\end{thebibliography}

\end{document}